\documentclass{amsart}
\usepackage{amssymb}
\usepackage{verbatim}
\usepackage{array}
\usepackage{latexsym}
\usepackage{enumerate}
\usepackage{amsmath}
\usepackage{amsfonts}
\usepackage{amsthm}
\usepackage{color}
\usepackage[english]{babel}

\newtheorem{theorem}{Theorem}[section]
\newtheorem{proposition}[theorem]{Proposition}
\newtheorem{lemma}[theorem]{Lemma}

\newtheorem{result}[theorem]{Result}

\newtheorem{rem}[theorem]{Remark}

\def\cC{\mathcal C}
\def\cD{\mathcal D}
\def\cE{\mathcal E}

\def\cQ{\mathcal Q}

\def\cP{\mathcal P}

\def\cX{\mathcal X}

\def\l{\ell}
\def\Aut{\mbox{\rm Aut}}
\def\K{\mathbb{K}}
\def\PG{{\rm{PG}}}
\def\ord{\mbox{\rm ord}}
\def\deg{\mbox{\rm deg}}

\def\Aut{\mbox{\rm Aut}}

\def\supp{\mbox{\rm Supp}}

\def\gg{\mathfrak{g}}

\newcommand{\PSL}{\mbox{\rm PSL}}
\newcommand{\PGL}{\mbox{\rm PGL}}
\newcommand{\AGL}{\mbox{\rm AGL}}
\newcommand{\PSU}{\mbox{\rm PSU}}
\newcommand{\PGU}{\mbox{\rm PGU}}
\newcommand{\SU}{\mbox{\rm SU}}
\newcommand{\Sz}{\mbox{\rm Sz}}

\newcommand{\aut}{\mbox{\rm Aut}}

\newcommand{\diag}{\mbox{\rm diag}}

\newcommand{\vf}{\varphi}

\def\supp{{\rm Supp}}

\newcommand{\ha}{{\textstyle\frac{1}{2}}}

\newcommand{\bA}{{\bf A}}

\newcommand{\bS}{{\bf S}}

\date{}

\begin{document}

\title{Curves with more than one inner Galois point}
\date{}
\author{G\'abor Korchm\'aros, Stefano Lia and Marco Timpanella}

\maketitle

\vspace{0.5cm}\noindent {\em Keywords}:
Algebraic curves, algebraic function fields, positive characteristic, automorphism groups.
\vspace{0.2cm}\noindent

\vspace{0.5cm}\noindent {\em Subject classifications}:
\vspace{0.2cm}\noindent  14H37, 14H05.


\begin{abstract}

 Let $\cC$ be an irreducible plane curve of $\PG(2,\mathbb{K})$ where $\mathbb{K}$ is an algebraically closed field of characteristic $p\geq 0$. A point $Q\in \cC$ is an inner Galois point for $\cC$ if the projection $\pi_Q$ from $Q$ is Galois. Assume that $\cC$ has two different inner Galois points $Q_1$ and $Q_2$, both simple. Let $G_1$ and $G_2$ be the respective Galois groups. Under the assumption that $G_i$ fixes $Q_i$, for $i=1,2$, we provide a complete classification of $G=\langle G_1,G_2 \rangle$ and we exhibit a curve for each such $G$. Our proof relies on deeper results from group theory.
\end{abstract}
\maketitle
    \section{Introduction}
In this paper, $\cX$ stands for a (projective, geometrically irreducible, non-singular) algebraic curve defined over an algebraically closed field $\mathbb{K}$ of characteristic $p\geq 0$. Also, $\cC$ stands for  a plane model of $\cX$, that is, for a plane curve $\cC$ defined over $\mathbb{K}$ and birationally equivalent to $\cX$. Let $\varphi$ be a morphism $\cX \mapsto \PG(2,\mathbb{K})$ which realizes it, so that $\varphi$ is  birational onto its image $\cC$. Further, $\mathbb{K}(\cX)$ denotes the function field of $\cX$, and $\aut(\cX)$ stands for the automorphism group of $\cX$ which fixes $\mathbb{K}$ element-wise. A point $Q$ in $PG(2,\mathbb{K})$ is a \emph{Galois point} for $\cC$ if the projection $\pi_Q$ from $Q$ is Galois; more precisely, if the field extension $\mathbb{K}(\cX)/\pi^*_Q(\mathbb{K}(\PG(1,\mathbb{K}))$ is Galois.
In this case,  if $G$ is the Galois group which realizes $\pi_Q$, then $Q$ \emph{is a Galois point with Galois group} $G$.
A Galois point $Q$ is either {\emph{inner}} or {\emph{outer}} according as $Q\in \cC$ or $Q\in PG(2,\mathbb{K})\setminus \cC$.
An inner Galois point may be a singular point of $\cC$.

The concept of a Galois point is due to H. Yoshihara and dates back to late 1990s; see \cite{yu}. Ever since, several papers have been dedicated to studies on Galois points, especially on the number of Galois points of a given plane curve. For non-singular plane curves, that number is already known \cite{fuka1,yu}. Nevertheless, for plane models with singularities the picture is much more involved, as it emerges from several recent papers \cite{fuka1,fuka2,fukas,fuka3,fukahase, FukaHiga,fukasp,hom,miu,yu1} where the authors focused on the problem of determining plane curves with at least two Galois points.

In this context, our paper is about plane models $\cC$ of $\cX$ with two different inner Galois points $\varphi(P_1)$ and $\varphi(P_2)$ both simple, or more generally unibranch. Here $\cC$ is \emph{unibranch} at its point $Q$ if $\varphi(P)=\varphi(R)=Q$ implies $P=R$.

Let $\varphi(P_1),\varphi(P_2)\in \cC$ be two different inner Galois points with Galois groups $G_1$ and $G_2$ respectively. Then
\begin{itemize}
\item[(I)] The quotient curves $\cX/G_1$ and $\cX/G_2$ are rational.
\end{itemize}
From now on we assume that $G_i$ fixes $P_i$, for $i=1,2$.
By Lemma \ref{profuka1} (see also \cite{fukas}), $\varphi(P_1)$ and $\varphi(P_2)$ are simple if the following two properties hold.
\begin{itemize}
\item[(II)] $G_1$ and $G_2$ have trivial intersection.
\item[(III)] In the divisor group of $\cX$,
$P_1+\sum_{\sigma\in G_1}\sigma(P_2)=P_2+\sum_{\tau\in G_2}\tau(P_1).$
\end{itemize}
Since (I),(II),(III) are independent of the model $\cC$, general properties of inner Galois points can be obtained by investigating curves $\cX$ with two subgroups $G_1,G_2\in \aut(\cX)$ satisfying (I),(II) and (III) with $|\supp(D|>2$. In this paper we go in that direction pursuing the strategy of using not only function field theory but also deeper results from group theory. Our starting point is to look inside the action of $G=\langle G_1,G_2\rangle$ on the support $\Omega$ of the divisor in (III). Lemma \ref{profuka1} shows that the action of $G_i$ on $\Omega\setminus\{P_i\}$ is sharply transitive, and hence $G$  induces on $\Omega$ a doubly transitively permutation group. Furthermore, a $1$-point stabilizer of $G$ is solvable.
It should be noticed that some non-trivial element of $G$ may fix $\Omega$ pointwise. In other words, the kernel $K$ of the permutation representation $\bar{G}$ of $G$ on $\Omega$ may be non-trivial so that $\bar{G}=G/K$ is the doubly transitive permutation group induced by $G$ on $\Omega$.
Since all doubly transitive permutation groups with solvable $1$-point stabilizer have been classified in 1970's by Holt \cite{holt} and O'Nan \cite{onan},  this gives a chance to determine the possibilities for $\bar{G}$  and then recover $G$ from $\bar{G}$ using Schur multipliers. In this strategy, an important simplification is that $G_1$ is a normal subgroup of the stabilizer of $P_1$ in $G$. Also, a natural idea is to regard $G$ as a doubly transitive group space on $\Omega$ where $G_1$ is a normal subgroup of a $1$-point stabilizer of $G$ and $G_1$ is sharply transitive on the remaining points of $\Omega$. Such doubly transitive group spaces were completely determined by Hering \cite{her}. It turns out that Hering's result provides a complete list of possibilities for $G$ and its action on $\Omega$. The question of which of these possibilities actually occur for some curve $\cX$ is completely answered in our main theorem.
\begin{theorem}
\label{th1}  Let $\cC$ be a plane model of $\cX$  associated with the morphism $\varphi: \cX\mapsto PG(2,\mathbb{K})$. Let $P_1,P_2\in \cX$ be two distinct points together with two distinct  subgroups $G_1,G_2$ of $\aut(\cX)$ such that $\varphi(P_1)$ and $\varphi(P_2)$ are simple Galois points of $\cC$ with Galois groups $G_1$ and $G_2$, respectively. If $G_i$ fixes $P_i$ for $i=1,2$ then $G=\langle G_1, G_2\rangle$ is isomorphic to one of the following groups:
\begin{itemize}
\item[(i)] $\PSL(2,q), {\rm{SL}}(2,q), Sz(q), \PSU(3,q), {\rm{SU}}(3,q), Ree(q)$ where $q$ is a power of $p$, and  $\deg(\cC)$ equals $q+1$ in the linear case, $q^{2}+1$ in the Suzuki case and $q^{3}+1$ in the unitary and Ree case. Here $G$ is supposed to be non-solvable.
\item[(ii)] $\rm{P\Gamma L}(2,8)$, $p=3$, and $\deg(\cC)=28$.
\item[(iiia)] ${\rm{AGL}}(1,m)$ for a prime power $m$ of $p$, $\deg(\cC)=m$, and $\cX$ is rational.
\item[(iiib)] ${\rm{AGL}}(1,3)$, $p\neq 3$, $\deg(\cC)=3$,  and $\cX$ is rational.
\item[(iiic)] ${\rm{AGL}}(1,4)$, $p\neq 2$, $\deg(\cC)=4$,  and $\cX$ is rational.
\item[(iva)] ${\rm{AGL}}(1,m)$, for $m=3,4,5,7$, $p\neq 2,3$, $\deg(\cC)=m$ and $\cX$ is elliptic.
\item[(ivb)] ${\rm{AGL}}(1,m)$, for $m=3,4,5,7$, $p=3$, $\deg(\cC)=m$, and $\cX$ is elliptic.
\item[(ivc)] ${\rm{PSU}}(3,2)$, $p=2$, $|\Omega|=9$, and $\cX$ is elliptic.
\item[(ivd)] ${\rm{AGL}}(1,m)$, for $m=3,5,7$, $p=2$, $\deg(\cC)=m$, and $\cX$ is elliptic.
\item[(ive)] $(C_5\times C_5)\times {\rm{SL}}(2,3)$, for $p=2$, $\deg(\cC)=25$, and $\cX$ is elliptic.
\item[(va)] ${\rm{SU}}(3,2)$, $p=2$, and $\gg(\cX)=10$.
\item[(vb)] ${\rm{SL(2,3)}}$, $p\neq 2,3$ and $\gg(\cX)=3$.
\end{itemize}
\end{theorem}
All the above cases occur, see Section \ref{esempi}. A corollary of Theorem \ref{th1} is the following result.
\begin{theorem}
\label{th2} Under the hypotheses of Theorem \ref{th1}, if $p\nmid |G_1|$, in particular if $p=0$ or $p>2\gg(\cX)+1$, then $\cX$ is either rational or elliptic, or it has genus $3$.
\end{theorem}
\begin{rem}
\label{rem13mar} If the order of the $1$-point stabilizer of any point in $G$ is coprime with $p$ (that is $G$ is tame), then the hypothesis that $\varphi(P_1)$ and $\varphi(P_2)$ are simple Galois points  can be relaxed to unibranch Galois points, with just one exception, namely
\begin{itemize}
\item[(via)] $G=G_1\times G_2$ is cyclic, $\deg(C)=|G_1|+|G_2|$, and $\gg(\cX)=0$.
\end{itemize}
{\em{For an example; see Remark \ref{rem25mar2020}.}}
\end{rem}

\begin{rem}{\em{ For a Galois point $\varphi(Q)$ with Galois group $G$ it may happen that $G$ does not fix any point $P\in\cX$ such that $\varphi(P)=Q$; an example is given in Remark \ref{ex19mar}.}}
\end{rem}

Our notation and terminology are standard. In particular, ${\rm{AGL}}(1,m)$ denotes the automorphism group of the affine line over $\mathbb{F}_m$. Here, ${\rm{AGL}}(1,3)\cong \bS_3$, ${\rm{AGL}}(1,4)\cong \bA_4$. 

\section{Background from function Field theory and some preliminary results }
For a subgroup $G$ of $\aut(\cX)$, let $\bar \cX$ denote a non-singular model of $\K(\cX)^G$, that is,
a projective non-singular geometrically irreducible algebraic
curve with function field $\K(\cX)^G$, where $\K(\cX)^G$ consists of all elements of $\K(\cX)$
fixed by every element in $G$. Usually, $\bar \cX$ is called the
quotient curve of $\cX$ by $G$ and denoted by $\cX/G$. The field extension $\K(\cX)|\K(\cX)^G$ is  Galois of degree $|G|$.

Since our approach is mostly group theoretical, we often use notation and terminology from finite group theory rather than from function field theory.

Let $\Phi$ be the cover of $\cX\mapsto \bar{\cX}$ where $\bar{\cX}=\cX/G$ is a quotient curve of $\cX$ with respect to $G$.
 A point $P\in\cX$ is a ramification point of $G$ if the stabilizer $G_P$ of $P$ in $G$ is nontrivial; the ramification index $e_P$ is $|G_P|$; a point $\bar{Q}\in\bar{\cX}$ is a branch point of $G$ if there is a ramification point $P\in \cX$ such that $\Phi(P)=\bar{Q}$; the ramification (branch) locus of $G$ is the set of all ramification (branch) points. The $G$-orbit of $P\in \cX$ is the subset
$o=\{R\in \cX \mid R=g(P),\, g\in G\}$, and it is {\em regular} (or long) if $|o|=|G|$, otherwise $o(P)$ is {\em short}. For a point $\bar{Q}$, the $G$-orbit $o$ lying over $\bar{Q}$ consists of all points $P\in\cX$ such that $\Phi(P)=\bar{Q}$. If $P\in o$ then $|o|=|G|/|G_P|$ and hence $\bar{\cQ}$ is a branch point if and only if $o$ is a short $G$-orbit. It may be that $G$ has no short orbits. This is the case if and only if every non-trivial element in $G$ is fixed--point-free on $\cX$, that is, the cover $\Phi$ is unramified. On the other hand, $G$ has a finite number of short orbits. For a non-negative integer $i$, the $i$-th ramification group of $\cX$
at $P$ is denoted by $G_P^{(i)}$ (or $G_i(P)$ as in \cite[Chapter
IV]{serre1979})  and defined to be
$$G_P^{(i)}=\{g\mid \ord_P(g(t)-t)\geq i+1, g\in
G_P\}, $$ where $t$ is a uniformizing element (local parameter) at
$P$. Here $G_P^{(0)}=G_P$.
The structure of $G_P$ is well known; see for instance \cite[Chapter IV, Corollary 4]{serre1979} or \cite[Theorem 11.49]{HKT}.
\begin{result}
\label{res74} The stabilizer $G_P$ of a point $P\in \cX$ in $G$ has the following properties.
\begin{itemize}
\item[\rm(i)] $G_P^{(1)}$ is the unique normal $p$-subgroup of $G_P$;
\item[\rm(ii)] For $i\ge 1$, $G_P^{(i)}$ is a normal subgroup of $G_P$ and the quotient group $G_P^{(i)}/G_P^{(i+1)}$ is an elementary abelian $p$-group.
\item[\rm(iii)] $G_P=G_P^{(1)}\rtimes U$ where the complement $U$ is a cyclic whose order is prime to $p$.
\end{itemize}
\end{result}
Let $\bar{\gg}$ be the genus of the quotient curve $\bar{\cX}=\cX/G.$ The Hurwitz
genus formula is the following equation
    \begin{equation}
    \label{eq1}
2\gg-2=|G|(2\bar{\gg}-2)+\sum_{P\in \cX} d_P.
    \end{equation}
    where
\begin{equation}
\label{eq1bis}
d_P= \sum_{i\geq 0}(|G_P^{(i)}|-1).
\end{equation}
Here $D(\cX|\bar{\cX})=\sum_{P\in\cX}d_P$ is the {\emph{different}}. For a tame subgroup $G$ of $\aut(\cX)$, that is for $p\nmid |G_P|$,
$$\sum_{P\in \cX} d_P=\sum_{i=1}^m (|G|-\ell_i)$$
where $\ell_1,\ldots,\ell_m$ are the sizes of the short orbits of $G$.

A group is a $p'$-group (or a prime to $p$ group) if its order is prime to $p$. A subgroup $G$ of $\aut(\cX)$ is {\em{tame}} if the $1$-point stabilizer of any point in $G$ is $p'$-group. Otherwise, $G$ is {\em{non-tame}} (or {\em{wild}}). Obviously, every $p'$-subgroup of $\aut(\cX)$ is tame, but the converse is not always true. From the classical Hurwitz's bound,
if $|G|>84(\gg(\cX)-1)$ then $G$ is non-tame; see  \cite{stichtenoth1973II} or {\cite[Theorems 11.56]{HKT}.
An orbit $o$ of $G$ is {\em{tame}} if $G_P$ is a $p'$-group for $P\in o$, otherwise $o$ is a {\em{non-tame orbit}} of $G$.

Let $\gamma$ be the $p$-rank of $\cX$, and let $\bar{\gamma}$ be the $p$-rank of the quotient curve $\bar{\cX}=\cX/G$.
The Deuring-Shafarevich formula, see \cite{sullivan1975} or \cite[Theorem 11,62]{HKT}, states for a $p$-subgroup $G$ of $\aut(\cX)$ that
\begin{equation}
    \label{eq2deuring}
\gamma-1={|G|}(\bar{\gamma}-1)+\sum_{i=1}^k (|G|-\ell_i)
    \end{equation}
where $\ell_1,\ldots,\ell_k$ are the sizes of the short orbits of $G$.
\begin{result}
\label{lem29dic2015} If $\cX$ has zero $p$-rank then $\aut(\cX)$ has the following properties:
\begin{itemize}
\item[\rm(i)] A Sylow $p$-subgroup of $\aut(\cX)$ fixes a point $P\in \cX$ but its nontrivial elements have no fixed point other than $P$.
\item[\rm(ii)] The normalizer of a Sylow $p$-subgroup fixes a point of $\cX$.
\item[\rm(iii)] Any two distinct Sylow $p$-subgroups have trivial intersection.
\end{itemize}
\end{result}
Claim (i) is \cite[Theorem 11.129]{HKT}. Claim (ii) follows from Claim (i). Claim (iii) is \cite[Theorem 11.133]{HKT}.

For the following results, see \cite[Lemmas 11.129, 11.75, 11.60]{HKT}
\begin{result}
\label{lem11.131HKT}
Assume that $\aut(\cX)$
contains a $p$-subgroup $G$ of order $p^r$. If the quotient curve $\cX/G$ has
$p$-rank zero$,$ and every non-trivial element in $G$ has exactly one fixed point, then $\cX$ has $p$-rank zero.
\end{result}
\begin{result}[Serre]
\label{lem11.75(i)}Let $\alpha\in G_{P}$ and $\beta\in
G_{P}^{(k)}, \, k\geq 1$. If $\alpha\not\in G_{P}^{(1)},$ then the commutator
$[\alpha,\beta]=\alpha\beta\alpha^{-1}\beta^{-1}$ belongs to $G_{P}^{(k+1)}$
if and only if either $\alpha^{k}\in G_{P}^{(1)}$ or $\beta\in
G_{P}^{(k+1)}$.
\end{result}
\begin{result}
\label{theorem11.60HKT}
If the order $n$ of $G_{P}$ is prime
to $p,$ then $n\leq 4\gg(\cX)+2.$
\end{result}
Let $\cE$ be a non-singular plane cubic curve viewed as a birational model of an elliptic curve $\cX$.
For an inflection point $O$ of $\cE$, the set of points of $\cE$ can be equipped by an operation $\bigoplus$ to form an abelian group $G_O$ with zero-element $O$, which is isomorphic to the zero Picard group of $\cE$; see for instance \cite[Theorem 6.107]{HKT}. The translation $\tau_a$ associated with $a\in \cE$ is the permutation on the points of $\cE$ with equation $\tau_a: x\mapsto x\bigoplus a$. Since there exists an automorphism in $\aut(\cE)$ which acts on $\cE$ as $\tau_a$ does, translations of $\cE$ can be viewed as elements of $\aut(\cE)$. They form the translation group $J(\cE)$ of $\cE$ which acts faithfully on $\cE$ as a sharply transitive permutation group. For every prime $r$, the elements of order $r$ in $J(\cE)$ are called $r$-torsion points. They together with the identity form an elementary abelian $r$-group of rank $h$. Here $h=2$ for $r\neq p$ while $h$ equals the $p$-rank of $\cE$ for $r=p$, that is, $h=0,1$ according as $\cE$ is supersingular or not.
\begin{result}
\label{res12feb2019} The translation group $J(\cE)$ is a normal subgroup of $\aut(\cE)$, and $\aut(\cE)=J(\cE)\rtimes \aut(\cE)_P$ for every $P\in \cE$.
\end{result}
\begin{proof} For complex cubic curves the claim is known. Here we provide a characteristic free proof based on \cite[Theorem 4.8]{silverman2009}. Let $O$ be the neutral element of the group structure of $\cE$. Then $\aut(\cE)_O$ is additive on $\cE$; see \cite[Theorem 4.8]{silverman2009}. Therefore $\aut(\cE)_O$ normalizes the group of translations $J(\cE)$. By transitivity of $J(\cE)$, $\aut(\cE) = J(\cE)\aut(\cE)_O$, and $J(\cE)\cap J(\cE)_O=\{{\rm{id}}\}$ by regularity of $J(\cE$) on $\cE$. Furthermore, again by transitivity of $J(\cE)$, $O$ may be replaced by any $P\in \cE$.
\end{proof}
The following result comes from \cite[Theorem 10.1]{silverman2009} and \cite[Theorem 11.94]{HKT}.
\begin{result}
\label{silv} Let $\cE$ be an elliptic curve, and $P\in\cE$. If the stabilizer $H$ of $P$ in $\aut(\cE)$ has order at least $3$ then
\begin{equation}
\begin{array}{lll}
 \mbox{$H\cong C_4$, or $H\cong C_6$} & \mbox{\quad when $p\ne 2,3$;}\\
 \mbox{$H\cong C_3\rtimes C_4$, and $j(\cE)=0$} & \mbox{\quad when $p=3$;}\\
 \mbox{$H\cong {\rm{SL}}(2,3)$ and $j(\cE)=0$} & \mbox{\quad when $p=2$.}
\end{array}
\end{equation}
\end{result}
If $j(\cE)=0$ then $\cE$ is birationally equivalent to either the cubic of affine equation $y^2=x^3+1$, $y^2=x^3-x$ or $y^2+y=x^3$, according as $p\neq 2,3$, $p=3$ or $p=2$.
Result \ref{silv} has the following corollary, see \cite[Theorem 11.94]{HKT}.
\begin{result}
\label{autelliptic} Let $\cE$ be an elliptic curve. If $G$ is a subgroup of $\aut(\cE)$ and $P\in \cE$ then
\begin{equation}
|G_P| = \left
\{
\begin{array}{ll}
 1,2,4,6 & \mbox{\quad when $p\neq 2,3,$} \\
 1,2,4,6,12 & \mbox{\quad when $p=3,$ }\\
 1,2,4,6,8,24 & \mbox{\quad when $p=2$}.
\end{array}
\right.
\end{equation}
Moreover, if $G_\cP$ is non-trivial then the quotient curve $\cE/G$ is rational.
For $p=2,$ the stabilizer $G_{\cP}$ is
cyclic when $|G_\cP|\leq 4,$ and it is the quaternion group
 when $|G_\cP|=8$, and the linear group ${\rm{SL}}(2,3)$ when $|G_\cP|=24$. All cases occur.  
\end{result}

For a prime $r$, let $R$ be the group of $r$-torsion points. Since $R$ is the unique elementary abelian $r$-subgroup of $J(\cE)$, and  $J(\cE)$ is a normal subgroup of $\aut(\cE)$, $R$ is also a normal subgroup of $\aut(\cE)$.
\begin{lemma}
\label{lemA10feb2019} Let $\cE$ be an elliptic curve, and $\alpha\in \aut(\cE)$ a non-trivial automorphism of prime order $t\neq p$. If $\alpha$ has at least two fixed points, then either $t=2$ and $\alpha$ has exactly $4$ fixed points, or $t=3$ and $\alpha$ has exactly $3$ fixed points.  Furthermore,
\begin{itemize}
\item[(i)] if $t=3$, no non-trivial translation of $J(\cE)$  preserving the set of fixed points of $\alpha$ has order $3$;
\item[(ii)] if $t=2$ and, in addition, $4$ divides the stabilizer of a fixed point of $\alpha$ then no non-trivial translation of $J(\cE)$  preserving the set of fixed points of $\alpha$ has order $2$.
\end{itemize}
 \end{lemma}
\begin{proof} The Hurwitz genus formula applied to the subgroup generated by $\alpha$ gives
$$0=2\gg(\cE)-2=-2t+\lambda(t-1)$$
where $\lambda$ counts the fixed points of $\alpha$. From this, the first claim follows. Let $t=3$. Since the $3$-torsion group $R$ of $\cE$ has order $9$, $\alpha$ together with $R$ generate a subgroup $M$ of $\aut(\cE)$ of order $27$. For a fixed point $P\in\cE$ of $\alpha$, let $\Delta$ be the $R$-orbit of $P$. As $R$ is a normal subgroup of $M$, $\Delta$ is left invariant by $M$. Furthermore,
since $|\Delta|=9$, the stabilizer $M_P$ of $P$ in $M$ has order $3$ and its three fixed points are in $\Delta$. Therefore, $M_P=\langle \alpha \rangle$ and the fixed points of $\alpha$ are in $\Delta$.
Since $J(\cE)$ is sharply transitive on $\cE$, this yields that no non-trivial element of order prime to $3$ may takes $P$ to another fixed point of $\alpha$ whence (i) follows for $t=3$. Let $t=2$. This time $R$ is an elementary abelian group of order $4$ which together with $\alpha$ generate a subgroup of $\aut(\cE)$ of order eight. Also, $M_P$ has order two and hence
again $M_P=\langle \alpha \rangle$, and $\alpha$ fixes either two points in $\Delta$, or all its $4$ fixed points are in $\Delta$. In the latter case, (ii) follows as for $t=3$. To investigate the former case, suppose that $\alpha=\gamma^2$ with $\gamma\in\aut(\cE)$ fixing $P$. The subgroup $T$ generated by $R$ together with $\gamma$ has order $16$ and preserves $\Delta$. The kernel of the representation of $T$ on $\Delta$ is not faithful, as $|\bS_4|$ is not divisible by $16$. Therefore, $T$ contains an involution $\tau$ fixing $\Delta$ pointwise. Since $P\in\Delta$ and the stabilizer of $P$ in  $\aut(\cE)$ is cyclic, $\tau$ coincides with $\alpha$ whence (ii) follows.
\end{proof}

 For a plane model $\cC$ of $\cX$ associated with the morphism $\varphi: \cX\mapsto PG(2,\mathbb{K})$, there exists a one-to-one correspondence between points of $\cX$ and branches of $\cC$. For any point $P\in\cX$ the associated branch $\gamma$ of $\cC$ is centered at $\varphi(P)$. Furthermore, the order of $\gamma$ is the positive integer $j_1$ such that the intersection number $I(\varphi(P),\gamma\cap \ell)=j_1$ for all but just one line through $\varphi(P)$. For the exceptional line $t$, called the tangent to $\gamma$ at $\varphi(P)$, we have  $I(\varphi(P),\gamma\cap t)=j_2$ with $j_2>j_1$; see \cite[Section 4.2]{HKT}.

  Let $\omega$ be the quadratic transformation with fundamental points $A_1A_2A_3$ and exceptional lines $A_1A_2$,$A_2A_3$, $A_3A_1$, where $A_1=\varphi(P_1),A_2=\varphi(P_2),A_3=t_1\cap t_2$, with $t_i$ the tangent line at $P_i$ ; see \cite[Sections 3.3, 3.4]{HKT}. For any non-exceptional line $\ell$ through a fundamental point $A_i$, the image of $\ell$ by $\omega$ is a line $\ell'$ through $A_i$; more precisely the points of $\ell$ distinct from $A_i$ are taken to the points of $\ell'$ distinct from $A_i$. For a branch $\delta$ of $\cC$ centered at a point $C$ of an exceptional line $A_iA_j$ with $C\neq \{A_i,A_j\}$, its image $\omega(\delta)$ is a branch centered at the opposite vertex $A_k$, and the tangent of $\omega(\delta)$ is a non-exceptional line through $A_k$. The converse also holds. If $C=A_i$ and $A_iA_j$ is the tangent of $\delta$, then $\omega(\delta)$ is a branch centered at $A_k$  and $A_kA_j$ is its tangent.

\begin{rem}
 \label{rem25mar20}{\em{ Let $P_1$ be an inner Galois point of $\cX$ with Galois group $G_1$. Up to a change of coordinates, $\cC$ has affine equation $f(X,Y)=0$, and $Y_\infty=\varphi(P_1)$. Furthermore, the $G_1$-fibers are represented by lines through $Y_\infty$.
 For a $G_1$-fiber $\Lambda$, let $\ell$ be such a line. Then a point $P\in\cX$ is in $\Lambda$ if and only if the associated branch $\gamma$ of $\cC$ has one of the following properties: either $\varphi(P)\neq \varphi(P_1)$ and $\varphi(P)\in \ell$, or $\varphi(P)=\varphi(P_1)$ and the tangent to $\gamma$ at $\varphi(P)$ coincides with the line $\ell$.
 Furthermore, if $G_1$ fixes $P_1$ then the fiber of $P_1$ contains no more point. Therefore, if $t$ is the tangent to $\cC$ at $\varphi(P_1)$, then $\gamma$ is the unique branch of $\cC$ whose center lies on $t$ and whose tangent coincides with $t$.}}
 \end{rem}

\begin{rem}
\label{ex19mar} {\em{The following example shows that $G_1$ may not fix any branch centered at $Y_\infty$.
For $p\neq 2$, let $\cX$ be a non-singular model of the singular plane curve with affine equation $Y^2=g(X)$ with a separable polynomial $g(X)\in \mathbb{K}[X]$. From \cite[Example 5.59]{HKT}, $\gg(\cX)=1$ and $Y_\infty$ is the unique singular point of $\cC$. More precisely, two branches  of $\cC$, say $\gamma$ and $\gamma'$, are centered at $Y_\infty$, both  tangent to the line $\ell_\infty$ at infinite. The linear map $u:\,(X,Y) \rightarrow  (X,-Y)$ is in $\aut(\cX)$, and $G_1=\langle u \rangle$ preserves every line $\ell$ through $Y_\infty$, acting transitively on its points distinct from $Y_\infty$. Therefore, $P_1=Y_\infty$ is an inner Galois point of $\cX$ with Galois group $G_1$ of order $2$, and the points $P_1,P_1'\in \cX$ associated with $\gamma,\gamma'$ respectively, form a $G_1$-fiber. We show that $G_1$ fixes no $P_1$ (and $P_1'$). Obviously, $G_1$ fixes each of the four points of $\cC$ lying on the $X$-axis. From the Hurwitz genus formula applied to $G_1$, $0=2\gg(\cX)-2=2(2\gg(\bar{\cX})-2))+n$
where $\bar{\cX}=\cX/G_1$ and $n$ is the number of fixed points of $G_1$ on $\cX$. Since $n\ge 4$, this is only possible for $\gg(\bar{\cX})=0$ and $n=4$. In particular, neither $Q_1$ nor $Q_2$ is fixed by $G_1$. Now suppose $p\neq 2,3$, and let $g(X)=\epsilon X(X-1)(X-\epsilon)(X-\epsilon^{2})$ for a primitive third root of unity $\epsilon$. A straightforward computation shows that $\cX$ has another inner Galois point, namely the origin $O=(0,0)$, with Galois group $G_2$ of order $3$ generated by the linear map  $v:\,(x,y) \rightarrow  (\epsilon x, \epsilon y)$. Since $O$ is not an inflection point with tangent $OY_\infty$, and $G_2$ fixes both $\gamma$ and $\gamma'$, the singletons $\{P_1\}$  and $\{P_2\}$ are $G_2$-fibers.
In particular, the line $OY_\infty$ contains no point from $\cC$ other than $P_1$ and $P_2$.
A generalization is obtained for $p\nmid d$ taking for $\cC$ the plane curve of affine equation $Y^d=g(X)$ with $g(X)=\epsilon X(X-1)(X-\epsilon)\cdots (X-\epsilon^{2d-2})$ where $\epsilon$ is a primitive $(2d-1)$th root of unity  }}
\end{rem}

From previous works on Galois points, we need a very recent result due to Fukasawa; see \cite[Theorem 1]{fukas}. We state it for the case of two inner Galois points $\varphi(P_1),\varphi(P_2)$. We also add some properties in case where the corresponding Galois group $G_i$ fixes $P_i$ for $i=1,2$.
\begin{lemma}
\label{profuka}
Let $\cC$ be a plane model of $\cX$  associated with the morphism $\varphi: \cX\mapsto PG(2,\mathbb{K})$. Let $P_1,P_2\in \cX$ be two distinct points together with two distinct subgroups $G_1,G_2$ of $\aut(\cX)$ such that $\varphi(P_1)$ and $\varphi(P_2)$ are unibranch Galois points of $\cC$ with Galois groups $G_1$ and $G_2$, respectively. Then the following properties hold:
\begin{itemize}
\item[(I)] The quotient curves $\cX/G_1$ and $\cX/G_2$ are rational;
\item[(II)] $G_1$ and $G_2$ have trivial intersection.
\end{itemize}

Assume in addition that $G_i$ fixes $P_i$ for $i=1,2$ and let $G=\langle G_1,G_2\rangle$.
 If
\begin{equation}
\label{eq14mar}
{\mbox{the line through $\varphi(P_1)$ and $\varphi(P_2)$ contains a further point of $\cC$}}
\end{equation}
then the stabilizer $(G_1)_{P_2}$ of $P_2$ in $G_1$ and the stabilizer $(G_2)_{P_1}$ of $P_1$
have the same order and that number equals the multiplicity of both $\varphi(P_1)$ and $\varphi(P_2)$.
Also, (\ref{eq14mar}) implies that the following conditions are equivalent:
\begin{itemize}
\item[\rm(III)]  In the divisor group of $\cX$,
$P_1+\sum_{\sigma\in G_1}\sigma(P_2)=P_2+\sum_{\tau\in G_2}\tau(P_1).$
\item[\rm(i)] Both $\varphi(P_1)$ and $\varphi(P_2)$ are simple points.
\item[\rm(ii)] Both $(G_1)_{P_2}$ and $(G_2)_{P_1}$ are trivial.
\end{itemize}
 For tame $G$, (\ref{eq14mar}) implies $\rm(ii)$.

If (\ref{eq14mar}) does not hold then either $\varphi(P_1)$ or $\varphi(P_2)$ is a singular point of $\cC$,  both $P_1$ and $P_2$ are fixed by $G$, and, for tame $G$, $G$ is cyclic.
\end{lemma}
\begin{proof}  By definition, (I) holds. Since both $G_1$ and $G_2$ are finite groups, and $\cC$ has a finite number of singular points, there exists a simple point $\varphi(P)\in \cC$ not on the line $\varphi(P_1)\varphi(P_2)$ which is not fixed by any non-trivial element from either $G_1$ or $G_2$. To show (II), assume by way of a contradiction that $g\in G_1\cap G_2$ with $g\neq 1$. Let $r$ be the line through $\varphi(P)$ and $\varphi(g(P))$ in $\PG(2,\mathbb{K})$. Then the points $\varphi(P_1),\varphi(P),\varphi(g(P))$ are three distinct points on the line $r$, and similarly,  $\varphi(P_2),\varphi(P),\varphi(g(P))$ are three distinct points on the same line $r$. This yields that $P$ lies on the line through $\varphi(P_1)$ and $\varphi(P_2)$, a contradiction.
Up to a change of the projective frame, $\varphi(P_1)=Y_\infty$ and $\varphi(P_2)=X_\infty$. Let $\mathcal{P}_1$ the set of all points of $\cX$  which are taken by $\varphi$ to points of $\cC$ lying on the line $\ell_\infty$ at infinity. Obviously, $P_2\in \mathcal{P}_1$, and hence the $G_1$-orbit $\Delta_1$ of $P_2$ is also contained in $\mathcal{P}_1$. Furthermore, every point $P\in \cX$ with
$\varphi(P)\neq \varphi(P_1)$ and $\varphi(P)\in \ell_\infty$ is in $\Delta_1$. However, a point $P\in\cX$ with $\varphi(P)=\varphi(P_1)$ is in $\Delta_1$ if and only if $\varphi(P)$, viewed as a branch of $\cC$ centered at $\varphi(P_1)$, is tangent to $\ell_\infty$. Let $\mathcal{Q}_1=\mathcal{P}_1\setminus \Delta_1$. In the divisor group of $\mathbb{K}(\cX)$, let $B_1=\sum_{P\in \mathcal{Q}_1}P$, $D_1=\sum_{P\in \Delta_1}P$, and $m_1=|(G_1)_{P_2}|=
|(G_1)_P|$ for every $P\in\Delta_1$. Now, since $\varphi(P_1)$ is unibranch, we have $B_1=P_1$, and hence $m_1(P_1+D_1)=
m_1P_1+\sum_{\sigma\in G_1}\sigma(P_2)$. On the other hand, since $P_1$ is a Galois point with Galois group $G_1$, in the intersection divisor $\cC\circ\ell_\infty$ the coefficient of $P\in \Delta_1$ is $|(G_1)_P|=m_1$. In particular,
$m_1$ is equal to the multiplicity of $\varphi(P_2)$.

The analog pointsets $\mathcal{P}_2,\Delta_2,\mathcal{Q}_2$ and divisors $B_2,D_2$ and $m_2$ are defined interchanging $P_1$ with $P_2$ and replacing $G_1$ by $G_2$.

Since (\ref{eq14mar}) implies that $|\Delta_1|>1$ and $|\Delta_2|>1$, their intersection contains a point $P$. Therefore, $m_1=m_2$. Let $m=m_1$. Then both points $\varphi(P_1)$ and $\varphi(P_2)$ have multiplicity $m$, and
$|G_1|=|G_2|=(\deg(\cC)-m)/m$. Therefore,
$$mP_1+\sum_{\sigma\in G_1}\sigma(P_2)=mP_2+\sum_{\tau\in G_2}\sigma(P_1).$$
Now, (III) holds if and only if $m=1$, that is, both points $\varphi(P_1)$ and $\varphi(P_2)$ are simple.
The latter condition is equivalent to $|(G_1)_{P_2}|=|(G_2)_{P_1}|=1$.

Since $|\Omega|>2$, $G$ is finite, otherwise $\cX$ would be either rational, or elliptic, and an infinite number of  elements in $G$ would fix $\Omega$ pointwise which contradicts Result \ref{silv} and the fact that no non-trivial automorphism of a rational curve may fix more than two points.
To show the final claim in Lemma \ref{profuka}, assume on the contrary that $m>1$, and take a point $Q\in\Delta_1\cap\Delta_2$. Then both $(G_1)_Q$ and $(G_2)_Q$ are subgroups of $G_Q$ of order $m$. Since $G$ is supposed to be tame, (iii) of Result \ref{res74} yields that $G_Q$ is cyclic whence $(G_1)_Q=(G_2)_Q$ follows. This contradicts (II).

Suppose that (\ref{eq14mar}) does not hold. Then B\'ezout's theorem, see \cite[Theorems 3.14, 4.36]{HKT}, applied to the line $\ell_\infty=\varphi(P_1)\varphi(P_2)$ yields $\deg(\cC)=\deg(\cC\circ \ell_\infty)=I(\varphi(P_1),\cC\cap \ell_\infty)+ I(\varphi(P_2),\cC\cap \ell_\infty)$. Since $\varphi(P_i)$ is unibranch and $\ell_\infty$ is not the tangent to $\cC$ at $P_i$, the multiplicity $\mu_i$ of $\varphi(P_i)$ equals $I(\varphi(P_i),\cC\cap \ell)$.
Therefore, $\deg(\cC)=\mu_1+\mu_2$. From $\deg(\cC)>2$, either $\mu_1$ or $\mu_2$ exceeds $1$, and hence one of the points $\varphi(P_1),\varphi(P_2)$ is singular. To show that $G_1$ fixes
$P_2$, it is enough to observe that $P_2$ is the unique pole of $y$ where $\mathbb{K}(y)=\cX^{G_1}$. Therefore, both $G_1$ and $G_2$ fix $P_2$, and this holds true for $P_1$ as $P_1$ is the unique pole of $x$ with $\mathbb{K}(x)=\cX^{G_2}$. Therefore $G$ fixes both $P_1$ and $P_2$. If $G$ is tame then (ii) Result \ref{res74} implies that $G$ is cyclic.
\end{proof}

\begin{rem}
\label{rem25mar2020} \em{ An example for the case where (\ref{eq14mar}) does not hold is the curve $f(X,Y)=X^uY^v-1$ with $u>v>1$ and ${\rm{g.c.d}}(u,v)=1$ where $u=|G_2|$ and $v=|G_1|$.
The automorphisms in $G_1$ are induced on $\cC$ by the homology $(X,Y)\mapsto (X,\lambda Y)$ with $\lambda$ ranging in the multiplicative subgroup of $\mathbb{K}$ of order $|G_1|$. The fixed points of such a homology in the plane are $Y_\infty$ and the points on the line $Y=0$. Therefore, a non-trivial automorphism in $G_1$ fixes exactly two points of $\cX$, namely $P_1$ and $P_2$. Now, the Hurwitz genus formula applied to $G_1$ gives $\gg(\cX)=0$. The same holds for $G_2$ and the group $G$ generated by $G_1$ and $G_2$ is the cyclic group of order $|G_1||G_2|$.
This gives case (via) in Remark \ref{rem13mar}. Earlier references for this example are \cite{fukamiura} and \cite{HY}.}
\end{rem}
}

\begin{lemma}
\label{profuka1} Let $P_1,P_2$ be two distinct points of $\cX$ together with two distinct subgroups $G_1,G_2$ of $\aut(\cX)$ such that $\rm{(I),(II),(III)}$ hold. Assume that $G_i$ fixes $P_i$ for $i=1,2$. Let $D$ be the divisor defined in $\rm(III)$. If $|\supp(D)|>2$ then
\begin{itemize}
\item[(i)] for $i=1,2$, the group $G_i$ is a sharply transitive group on $\supp(D)\setminus \{P_i\}$;
\item[(ii)] the group $G$ generated by $G_1$ and $G_2$ acts on $\supp(D)$ as a doubly transitive permutation group;
\end{itemize}

Furthermore, there exists a birational model $\cC$ of $\cX$ such that $\varphi(P_1)$ and $\varphi_(P_2)$ are Galois points with Galois groups $G_1$ and $G_2$ respectively, and
the equation $f(X,Y)=0$ of $\cC$ can be chosen in such way that
\begin{itemize}
\item[(iii)] $|\supp(D)|=\deg(\cC)$ and both $\varphi(P_1)$ and $\varphi(P_2)$ are simple points.
\item[(iv)] $\cX^{G_1}=\mathbb{K}(x)$ and $\cX^{G_2}=\mathbb{K}(y)$,
\item[(v)] $\varphi(P_1)=Y_\infty$ and $\varphi(P_2)=X_\infty$,
\item[(vi)] the poles of $x$ are the points in $\supp(D)\setminus \{P_1\}$, each of multiplicity $1$, and the poles of $y$ are the points in $\supp(D)\setminus \{P_2\}$, each of multiplicity $1$.
\end{itemize}
\end{lemma}
\begin{proof} Take $u,v\in \mathbb{K}(\cX)$ with $\cX^{G_1}=\mathbb{K}(u)$ and  $\cX^{G_2}=\mathbb{K}(v)$. Let $g(X,Y)\in \mathbb{K}[X,Y]$ be an irreducible polynomial such that $g(u,v)=0$. From \cite[Proposition 1]{fukas}, the plane curve $\cD$ with affine equation $g(X,Y)=0$ is a birational model of $\cX$. Then $X_\infty=(1:0:0), Y_\infty=(0,1,0)$ are Galois points of $\cD$ with Galois groups $G_1$ and $G_2$, respectively.  Let $\psi: \cX\mapsto \cD \subset \PG(2,\mathbb{K})$ be the associated morphism. For $i=1,2$, let $\gamma_i$ be the branch of $\cD$ associated with $P_i$. From Remark \ref{rem25mar20}, the tangent $t_i$ of $\gamma_i$ is different from the line $\psi(P_1)\psi(P_2)$. Let $\varphi=\omega \circ \psi$ where $\omega$ is a quadratic transformation  with fundamental points $U_1=\psi(P_1),U_2=\psi(P_2),U_0=t_1\cap t_2$,  and look at the birationally plane model $\cC$ associated to $\varphi$. An equation of $\cC$ is $f(X,Y)=0$ with $f(\omega(u),\omega(v))=0$.  From the properties of $\omega$ quoted before Remark \ref{rem25mar20}, both $U_2$ and $U_1$ are inner Galois points of $\cC$ with Galois group $G_2$ and $G_1$, respectively;
see also \cite{miu2}. 
Furthermore, from Remark \ref{rem25mar20}, both these points of $\cC$ are unibranch as $\omega(\gamma_1)$, $\omega(\gamma_2)$ are the unique branches of $\cC$ centered at $U_2$ and $U_1$ respectively. Also, the tangents of $\omega(\gamma_i)$ and $\omega(\gamma_2)$ are the lines $U_0U_2$ and $U_0U_1$ respectively.

Up to a change of $x$ by $x-a$ with $a\in\mathbb{F}^*$, $P_1$ is a pole of $x$ of multiplicity $1$. A similar change in $y$ ensures that $P_2$ is a pole of $y$ of multiplicity $1$. Thus (iv) and (v) hold.  Note that for $\sigma\in G_1$, each point $\sigma(P_2)$ is also a pole of $x$.

Now, Lemma \ref{profuka} applies. Since $|\supp(D)|>2$, (III) yields that Condition (\ref{eq14mar}) is satisfied. Therefore, $\varphi(P_1)$ and $\varphi(P_2)$ are simple points.

We point out that $\sigma(P_2)=P_2$ with $\sigma\in G_1$ only occurs when $\sigma=1$.
(III) reads
$$P_1+\sum_{\sigma\in G_1^*}\sigma(P_2)=\sum_{\tau\in G_2}\tau(P_1)$$ where $G_1^*$ denotes the set of non-trivial elements of $G_1$. Now, if $\sigma(P_2)=P_2$  with $\sigma\in G_1^*$ then $P_2$ would be in the support of the divisor on the left hand side, but not on the right hand side as $\tau(P_2)=P_2$ for every $\tau\in G_2$; a contradiction. Similarly $\tau(P_1)=P_1$ never holds for $\tau \in G_2^*$. Therefore (i) and hence (ii) follow from (III).  Also, $|\supp(D)|-1=|G_1|=|G_2|$. A further consequence is that the poles of $x$ are exactly the points $\supp(D)\setminus \{P_1\}$ each with multiplicity $1$. The same holds for $y$ when $P_1$ is replaced by $P_2$. From this (vi) follows.

Finally, since $\varphi(P_1)$ is a simple point of $\cC$, $|\supp(D)|=\deg(\cC)$ follows from (III).

\end{proof}

Assume that $P$ is a pole of $v\in \mathbb{K}(\cX)$ with multiplicity $1$. For a local parameter $t$ of $P$, we have $v=t^{-1}+w$ with $v_P(w)\geq 0$. If $\alpha\in \aut(\cX)$ fixes $P$ choose the smallest integer $m$ such that $\alpha^m(v)=v$. Assume that $m$ is a power of $p$ then $\alpha(v)=(t+\bar{w})^{-1}+w_1$ with $v_P(\bar{w})\geq 2$ and $v_P(w_1)\geq 0$. Since $(t+\bar{w})^{-1}=t^{-1}(1+w_2)$ with $v_P(w_2)\geq 1$ this yields
$v_P(\alpha(v)-v)\geq 0$, that is, $P$ is not a pole of $\alpha(v)-v$. For $p\nmid m$, the above argument can be adapted, as $ (ut+w)^{-1}=u^{-1}t^{-1}(1+w_3)$ with $v_P(w_3)\geq 0$. It turns out that $P$ is not a pole of $\alpha(v)-u^{-1}v$. This holds true for $\alpha^k$ when $u^{-1}$ is replaced by $u^{-k}$. Therefore, $P$ is not a pole of $\alpha(v)-u^{-1}v$ for
any $m$-th root of unity. This gives the following result.
\begin{lemma}
\label{lemA1jan} For a pole $P$ of $v\in \mathbb{K}(\cX)$, let $\alpha\in \aut(\cX)$ be a non-trivial automorphism fixing $P$. Let $m$ be the smallest integer such that $\alpha^m(v)=v$. If $m$ is a power of $p$ then $P$ is not a pole of $\alpha(v)-v$. If $p\nmid m$ then $P$ is not a pole of $\alpha(v)-uv$ for all $m$-th roots of unity $u\in \mathbb{K}$.
\end{lemma}
The following result is well known for complex curves; see \cite[Theorem 5.9]{acc}. It remains valid in any characteristic; see \cite[Theorem 11.114]{HKT}.
\begin{result}
\label{th11.114}  Let $S$ be a subgroup of $\Aut(\cX)$ of order $n$
which has a partition with components $S_1,\ldots,S_k,$ with $n_i=
|S_i|$\, for $i=1,\ldots, k,$ and  let $\gg',\,\gg'_i$ be the genera of the quotient curves
$\cX/S$ and $\cX/S_i$, for $i=1,\ldots,k$. Then
\begin{equation}
\label{part}
(k-1)\gg(\cX)+n\gg'=\sum_{i=1}^k\ n_i \gg'_i.
\end{equation}
\end{result}
\section{Background from group theory}
From group theory we need properties of Lie type simple groups, namely the  projective special group, the projective special unitary group, the Suzuki group $Sz(q)$, and the Ree Group $Ree(q)$. The main reference is \cite[Section 3]{wil}; see also \cite[Appendix A]{HKT}. Our notation and terminology are standard. In particular, $Z(G)$ stands for the center of a group $G$. The normal closure  $S$ of subgroup $H$ of a group $G$ is the subgroup generated by all conjugates of $H$  in $G$. By definition, $S$ is the smallest normal subgroup of $G$ containing $H$.

For  $q=r^h$ with $r$ prime, the projective special group $\PSL(2,q)$ has order $(q+1)q(q-1)/\tau$ with $\tau={\rm{g.c.d.}}(2,q+1)$. $\PSL(2,q)$  is simple for $q\ge 4$, isomorphic to a subgroup of the projective line $\PG(1,q)$ over $\mathbb{F}_q$ and doubly-transitive on the set $\Omega$ of points of $\PG(1,q)$. If $r=2$ then $\PGL(2,q)=\PSL(2,q)$ whereas, for $r$ odd,  $x\to (ax+b)/(cx+d) \in \PSL(2,q)$ if and only if $ad-bc$ is a non-zero square element of $\mathbb{F}_q$.
\begin{result}([Dickson's classification; see \cite[Theorem 3]{maddenevalentini1982})
\label{lem30oct2016}
The finite subgroups of the group $\PGL(2,\mathbb{K})$ are isomorphic to one of the following groups:
\begin{enumerate}
\item[\rm(i)] prime to $p$ cyclic groups;
\item[\rm(ii)] elementary abelian $p$-groups;
\item[\rm(iii)] prime to $p$ dihedral groups;
\item[\rm(iv)] Alternating group $\bA_4$;
\item[\rm(v)] Symmetric group $\bS_4$; and $p>2$
\item[\rm(vi)] Alternating group $\bA_5$;
\item[\rm(vii)] Semidirect product of an elementary abelian
$p$-group of order $p^h$ by a cyclic group of order $n>1$ with
 $n\mid(p^h-1);$
\item[\rm(viii)] $\PSL(2,p^f)$ for $f \mid m$;
\item[\rm(ix)] $\PGL(2,p^f)$ for $f \mid m$.
\end{enumerate}
\end{result}
Here, $\bA_4\cong {\rm{AGL}}(1,4)$, and $\bA_5\cong \PSL(2,5)$.

The special linear group ${\rm{SL}}(2,q)$ has center of order $2$, and  ${\rm{SL}}(2,q)/ Z({\rm{SL}}(2,q))\cong \PSL(2,q)$.
Moreover, the automorphism group of $\PSL(2,q)$ is the semilinear group $\rm{P\Gamma L}(2,q)$. Since $Z(\PSL(2,q))$ is trivial, $\PSL(2,q)$ can be viewed as a (normal) subgroup of $\rm{P\Gamma L}(2,q)$  consisting of all semilinear maps $x\to (ax^{\sigma}+b)/(cx^{\sigma}+d)$ where $a,b,c,d \in \mathbb{F}_q$ with $ad-bc\neq 0$, and $\sigma\in\aut(\mathbb{F}_q)$. The quotient group $\rm{P\Gamma L}(2,q)/\PSL(2,q)$ is either $C_h$, or $C_h\times C_2$, according as $r=2$, or $r$ is odd. The ``linear subgroup'' of $\rm{P\Gamma L}(2,q)$ is $\PGL(2,q)$ which is isomorphic to $\aut(\PG(1,q))$, and consists of  all linear maps $x\to (ax+b)/(cx+d)$ where $a,b,c,d \in \mathbb{F}_q$ with $ad-bc\neq 0$. Either $\PGL(2,q)=\PSL(2,q)$ or $[\PGL(2,q):\PSL(2,q)=2]$ according as $r=2$ or $r$ is odd.
\begin{lemma}
\label{lem11jan} Let $S_r$ be a Sylow $r$-subgroup of the $1$-point stabilizer $M$ of a subgroup $L$ of $\rm{P\Gamma L}(2,q)$ containing $\PSL(2,q)$. If $S_r$ contains a Sylow $r$-subgroup $T_r$  of $\PSL(2,q)$ then either $S_r=T_r$, or $r|h$ and $S_r$ is not a normal subgroup of $M$. Furthermore, if $S_r=T_r$ and $M/S_r$ is cyclic then $G\le \PGL(2,q)$.
\end{lemma}
\begin{proof} If $r\nmid h$ then the Sylow $r$-subgroups of $\PSL(2,q)$ are also the Sylow $r$-groups of $\rm{P\Gamma L}(2,q)$. Therefore, we may assume that $h=r^uv$ with $u\ge 1, r\nmid v$. Any Sylow $r$-subgroup $S_r$ of $\rm{P\Gamma L}(2,q)$ has order $qr^u$. Up to conjugacy, the $1$-point stabilizer is the subgroup of $\rm{P\Gamma L}(2,q)$ fixing the point at infinity $\infty$ of $\PG(1,q)$.
Then $T_r$ consists of all transformations $x\to x+b$ with $b\in \mathbb{F}_q$. Furthermore, the transformations $x\to x^{\sigma}+b$ with $\sigma\in \aut(GF(q))$, $\sigma^{r^u}=1$, and $b\in GF(q)$, form a group of order $qr^u$ which is a Sylow $r$-subgroup $F$ of $\rm{P\Gamma L}(2,q)$. By Sylow's theorem, $S_r$ may be assumed to be a subgroup of $F$. Let $w\in S_r$ be the semilinear transformation $w:x\to x^\sigma+a$ with a non-trivial automorphism $\sigma$ of order $p^k$ with $1\le k \le u$, and $a\in \mathbb{F}_q$. Take an element $\lambda\in \mathbb{F}_q$ of order $q-1$ for $q$ even and of order $\ha(q-1)$ for $q$ odd. Let $l(x)=\lambda x$. Then $l\in \PSL(2,q)$ and  $l$ fixes $\infty$. Also, $(l^{-1}wl)(x)=\lambda^{\sigma-1}x^\sigma+\lambda^{-1}a$. By way of contradiction, assume that $S_r$ is a normal subgroup of $M$. Then $l^{-1}wl\in S_r$ which yields $\lambda^{\sigma}=\lambda$, that is, $\lambda$ lies in a proper subfield $\mathbb{F}_{r^k}$ of $\mathbb{F}_q$. But this contradicts the choice of $\lambda$. Finally, if $S_r=T_r$ and $M/S_r$ is cyclic but $G\lvertneqq \PGL(2,q)$, let $M=S_r\rtimes U$ and take a semilinear transformation $u: x\rightarrow \lambda x^\sigma$ in $U$ together with a linear transformation $v: x\rightarrow \mu x$ such that $\mu^\sigma \neq \mu$.  Then $uv\neq vu$, and hence $U$ cannot be cyclic.
\end{proof}

For  $q=r^h$ with $r$ prime, the projective special unitary group $\PSU(3,q)$ has order $(q^3+1)q^3(q^2-1)/\mu$ with $\mu={\rm{g.c.d.}}(3,q+1)$. $\PSU(3,q)$ is simple for $q\ge 3$, isomorphic to a subgroup of $\aut(\mathcal{H}_q)$ and doubly-transitive on the set $\Omega$ of all $\mathbb{F}_{q^{^2}}$-rational points of $\mathcal{H}_q$. Furthermore, its automorphism group is the semilinear group $\rm{P\Gamma U}(3,q)$. Since $Z(\PSU(3,q))$ is trivial, $PSU(3,q)$ can be viewed as a (normal) subgroup of $\rm{P\Gamma U}(3,q)$. The ``linear subgroup'' of $\rm{P\Gamma U}(3,q)$ is $\PGU(3,q)$ which is isomorphic to $\aut(\mathcal{H}_q)$. Let $\infty$ denote the (unique) point at infinity $\infty$ of $\mathcal{H}_q$. Then the stabilizer of $\infty$ in $\rm{P\Gamma U}(3,q)$  consists of all transformations $t$ where $t(x)=ax^\sigma+c,t(y)=by^\sigma +\bar{a}^\sigma x+d$ with $a,b,c,d\in \mathbb{F}_{q^2}$, $\bar{a}=a^q, b\in \mathbb{F}_q^*, d^q+d=c^{q+1}$, and $\sigma\in \aut(\mathbb{F}_{q^2})$. Here $t\in PSU(3,q)$ for $\sigma=1$ and $a^m=1$ where either $m=\frac{1}{3}(q+1)$ or $m=q+1$, according as $3$ divides $q+1$ or does not.

The special unitary group $\SU(3,q)$ has center of order $\mu={\rm{g.c.d.}}(3,q+1)$, and  $\SU(3,q)/ Z(\SU(3,q))\cong \PSU(3,q)$.
\begin{lemma}
\label{lem11Ajan} Let $S_r$ be a Sylow $r$-subgroup $S_r$ of a $1$-point stabilizer $M$ of a subgroup $L$ of $\rm{P\Gamma U}(3,q)$ containing $\PSU(3,q)$. If $S_r$ contains a Sylow $r$-subgroup $T_r$  of $\PSU(3,q)$. Then either $S_r=T_r$, or $r|h$ and $S_r$ is not a normal subgroup of $M$. Furthermore, if $S_r=T_r$ and $M/S_r$ is cyclic then $G\le \PGU(3,q)$.
\end{lemma}
\begin{proof} We argue as in the proof of Lemma \ref{lem11jan}. By way of contradiction, $S_r$ may be assumed to contain a transformation $w$ where $w(x)=x^{\sigma}+a,w(y)=y^{\sigma}+x^{\sigma}+b$ with $b=a^q+a$ and $\sigma\in \aut(\mathbb{F}_{q^2})$  of order $r^k$ with $1\le k \le u$. Let $l$ be a transformation with $l(x)=\lambda x, l(y)=y$ where $\lambda \in \mathbb{F}_{q^2}$ has order $q+1$ for $3\nmid (q+1)$ and $\frac{1}{3}(q+1)$ for $3\mid (q+1)$. Then $l \in PSU(3,q)$ and $\l$ fixes $\infty$. Moreover,  $(l^{-1}wl)(x)=\lambda^{\sigma-1}x^\sigma+\lambda^{-1}a$. As in the proof of Lemma \ref{lem11jan}, this leads to a contradiction. For the proof of the final claim the argument in the proof of Lemma \ref{lem11jan} can be used.
\end{proof}
For $q=2^h$ with $h\geq 3$ odd, the Suzuki group $Sz(q)$ has order $(q^2+1)q^2(q-1)$. It is a simple group, isomorphic to  $\aut(\mathcal{S}_q)$ where $\mathcal{S}_q$ stands for the Suzuki curve, see \cite[Section 12.2]{HKT}. $Sz(q)$ acts faithfully as a doubly transitive permutation group on the set $\Omega$ of all
$\mathbb{F}_q$-rational points of $\mathcal{S}_q$. As $Z(Sz(q))$ is trivial, $Sz(q)$ can be viewed as a normal subgroup of its automorphism group $\aut(Sz(q))$. Furthermore, the quotient group $\aut(Sz(q))/Sz(q)$ is $C_h$. Therefore, the first claim of Lemma \ref{lem11jan} trivially holds for $r=2$ when $\PSL(2,q)$ and $\rm{P\Gamma L}(2,q)$ are replaced by $Sz(q)$ and $\aut(Sz(q))$, respectively.  A direct computation similar to that carried out at the end of the proof of Lemma \ref{lem11jan} shows that if $S_r=T_r$ and $M/S_r$ is cyclic then $G\le Sz(q)$.
\begin{lemma}
\label{lem11Cjan} Let $S_2$ be a Sylow $2$-subgroup of the $1$-point stabilizer $M$ of a subgroup $L$ of $\aut(Sz(q))$ containing $Sz(q)$. Then $S_r=T_r$. Furthermore, if  $M/S_r$ is cyclic then $G\le Sz(q)$.
\end{lemma}
For  $q=3^h$ with $h\ge 3$ odd, the Ree group $Ree(q)$ has order $(q^3+1)q^3(q-1)$. It is simple, isomorphic to $\aut(\mathcal{R}_q)$ and doubly-transitive on the set $\Omega$ of all $\mathbb{F}_q$-rational points of the Ree curve $\mathcal{R}_q$. As $Z(Ree(q))$ is trivial, $Ree(q)$ can be viewed as a normal subgroup of its automorphism group $\aut(Ree(q))$. Furthermore, the quotient group $\aut(Ree(q))/Ree(q)$ is $C_h$. Furthermore, $Ree(q)$ has a faithful representation in the six-dimensional projective space $\PG(6,q)$ as a subgroup of $\PGL(7,\mathbb{F}_q)$ which preserves the Ree-Tits ovoid $Q$. The action of $Ree(q)$ on $Q$ is doubly transitive, and it is the same as on $\Omega$. We refer to an explicit presentation of $Q$ in a projective frame $(X_0,X_1,\ldots,X_6)$ of $\PG(6,\mathbb{F}_q)$ as given in \cite[Appendix A, Example A.13]{HKT}.Then $Z_\infty=(0,0,0,0,0,1)\in Q$. Moreover, a Sylow $3$-subgroup $T_3$ of $Ree(q)$ fixes $Z_\infty$ and consists of all projectivities $\alpha_{a,b,c}$  associated to the matrices
$$
\left[\begin{array}{ccccccc}
1 & 0 & 0 & 0 & 0 & 0 & 0 \\ a & 1 & 0 & 0 & 0 &
0 & 0
\\ b & a^{\vf} & 1 & 0 & 0 & 0 & 0
\\ c & b-a^{\vf+1} & -a & 1 & 0 & 0 & 0
\\ v_1(a,b,c) & w_1(a,b,c) & -a^2 & -a & 1 & 0 & 0
\\ v_2(a,b,c) & w_2(a,b,c) & ab+c & b & -a^{\vf} & 1 & 0
\\ v_3(a,b,c) & w_3(a,b,c) & w_4(a,b,c) & c & -b+a^{\vf+1}& -a & 1
\end{array}
\right]
$$
for $a,b,c\in\mathbb{F}_q$. Also,
 the stabilizer $Ree(q)_{Z_{\infty},O}$ with
$O=(1,0,0,0,0,0,0)\in Q$ is the cyclic group $C_{q-1}$ consisting of projectivities
$\beta_d$ associated to the diagonal matrices,
$$
\diag(1,d,d^{\vf+1},d^{\vf+2},d^{\vf+3},d^{2\vf+3},
d^{2\vf+4})$$
for $d\in\mathbb{F}_q$.
The stabilizer of $Z_\infty$ in $Ree(q)$ is the semidirect product of $T_3\rtimes C_{q-1}$. Moroever, the stabilizer of $Z_\infty$ in $\aut(Ree(q))$ consists of all semilinear transformations which are products $uv$ where $u\in S_3$ and $v$ is a $\sigma$-Frobenius map of $\PG(6,\mathbb{F}_q)$ where, for every $\sigma\in \aut(\mathbb{F}_q)$, the  associated $\sigma$-Frobenius map is defined by $(X_0,\ldots,X_6)\to (X_0^\sigma,\ldots,X_6^\sigma)$. A direct computation similar to that carried out at the end of the proof of Lemma \ref{lem11jan} shows that if $S_r=T_r$ and $M/S_r$ is cyclic then $G\le Ree(q)$.
\begin{lemma}
\label{lem11Bjan} Let $S_3$ be a Sylow $3$-subgroup of the $1$-point stabilizer $M$ of a subgroup $L$ of $\aut(Ree(q))$ containing $Ree(q)$. If $S_3$ contains a Sylow $3$-subgroup $T_3$  of $Ree(q)$ then either $S_3=T_3$, or $3|h$ and $S_3$ is not a normal subgroup of $M$. Furthermore, if $S_r=T_r$ and $M/S_r$ is cyclic then $G\le Ree(q)$.
\end{lemma}
\begin{proof} We argue as in the proofs of Lemmas \ref{lem11jan} and \ref{lem11Ajan}. We may assume $h=3^uv$ with $3\nmid v$. Let $H_\infty$ be the hyperplane at infinity of equation $X_0=0$ so that the arising affine space $AG(6,\mathbb{F}_q)$ has coordinates $x_1=X_1/X_0,\ldots, x_6=X_6/X_0$. Look at the $1$-point stabilizer of $Z_\infty$. Up to an isomorphism, $S_3$ consists of products $\alpha\beta$ where $\alpha\in T_3$ and  $\beta$ is a Frobenius map $(x_1,\ldots,x_6)\to (x_1^\sigma,\ldots,x_6^\sigma)$ with $\sigma\in\aut(\mathbb{F}_q)$. In particular, $S_3$ contains a transformation $w$ such that  $w(x)=x^{\sigma}+a$, and $\sigma$  of order $3^k$ with $1\le k \le u$. For a primitive element $\lambda \in \mathbb{F}_q$, let $l$ denote a transformation associated with the diagonal matrix $\diag(1,\lambda,\lambda^{\vf+1},\lambda^{\vf+2},\lambda^{\vf+3},\lambda^{2\vf+3},
\lambda^{2\vf+4})$. Computing $l^{-1}wl(x)$ shows again that $l^{-1}wl\not \in S_3$,  a contradiction as in the proof of Lemma \ref{lem11jan} where $\infty$ is replaced with $Z_\infty$.
\end{proof}
 Essential tools in our work are the classification of finite 2-transitive permutation groups whose $1$-point stabilizer has a solvable normal subgroup due to Holt and O'Nan, and its generalization to group spaces, due to Hering.
 \begin{result}(Holt, \cite[Main Theorem]{holt})
\label{holtre} Let $G$ be a finite $2$-transitive permutation group of even degree, and suppose that the $1$-point stabilizer of $G$ is solvable. Then either $G$ has a regular normal subgroup, or $G$ has a normal $2$-transitive subgroup W isomorphic to $\PSL(2,q)$, $\PSU(3,q)$ (for some odd prime power $q$), or to $Ree(q)$. In the latter case, the action of $W$ is the natural $2$-transitive permutation representation of $\PSL(2,q), \PSU(3,q)$ and $Ree(q)$ respectively, with only one exception: $G\cong \rm{P}\Gamma L(2,8)$ and $W\cong \PSL(3,2)\cong \PSL(2,7)$ with degree $28$.
\end{result}
\begin{result}(O'Nan, \cite[Theorem B]{onan})
\label{onanre} Let $G$ be a finite $2$-transitive group of odd degree, and suppose that the $1$-point stabilizer $G$ has an abelian normal subgroup of order $>1$. Then $G$ has either a regular normal subgroup, or a normal $2$-transitive subgroup W isomorphic to
\begin{itemize}
\item[(i)] $\PSL(r+1,q)$, with  $1 + q +\ldots + q^r$ odd and $r\ge 1$, or
\item[(ii)] $\PSU(3,2^k)$, or
\item[(iii)] $\Sz(2^{2k+1})$,
\end{itemize}
and the action of $W$ is the natural $2$-transitive representation of $\PSL(r+1,q)$, $\PSU(3,2^k)$ and $\Sz(2^{2k+1})$, respectively.
\end{result}
A \emph{group space} consists of a pair $(\Omega,G)$ where $\Omega$ is a set and $G$ is generated, as an abstract group, by a set of permutations on $\Omega$. Clearly, $G$ induces a permutation group $\bar{G}$ on $\Omega$ so that $\bar{G}\cong G/K$ where the subgroup $K$ is the kernel consisting of elements in $G$ which fix $\Omega$ element-wise. A group space is {\emph{transitive}}, if $\bar{G}$ is transitive on $\Omega$. A transitive group space whose $1$-point stabilizer has a subgroup transitive on the remaining points is $2$-transitive.
\begin{result}(Hering, \cite[Theorem.2.4]{her})
\label{HC} Let $(\Omega, G)$ be a finite transitive group space with $|\Omega|>2$. Assume that for some $P\in\Omega$ the stabilizer $G_P$ contains a normal
subgroup $Q$ which is sharply transitive on $\Omega\setminus \{P\}$. If $S$ is the normal
closure of $Q$ in $G$, then one of the following holds:
\begin{itemize}
\item[(i)] $S\cong \PSL(2,q), {\rm{SL}}(2,q), Sz(q), \PSU(3,q), {\rm{SU}}(3,q), Ree(q)$, where $q$ is a prime power, and  $|\Omega|$ is $q+1$ in the linear case, $q^{2}+1$ in the Suzuki case and $q^{3}+1$ in the unitary and Ree case.
\item[(ii)] $S\cong {\rm{P\Gamma L}}(2,8)$ and $|\Omega|=28$.
\item[(iii)] $S$ is a sharply doubly transitive permutation group on $\Omega$.
\item[(iv)] $|\Omega|=d^{2}$ for $d\in\{3,5,7,11,23,29,59\}$, $S=O_{d}(S)\rtimes Q$, $O_{d}(S)$ is extraspecial
of order $d^{3}$ and exponent $d$, $Z(O_{d}(S))=Z(S)$ is the kernel of $(\Omega, S)$ and
$S$ induces a sharply 2-transitive group on $\Omega$.
\end{itemize}
\end{result}

To deal with Case (iii), we need a corollary to Zassenhaus' classification of finite sharply doubly transitive groups.
\begin{result}(Zassenhaus \cite[XII Theorem 9.8]{huppert3})
\label{sharply} Let $G$ be a sharply doubly transitive permutation group on a finite set $\Omega$. Then $|\Omega|$ is a prime power $m$, and the elements in $G$ which have no fixed point in $\Omega$ together with the identity permutation form an elementary abelian group $M$ of order $m$. An example is the group $\AGL(1,m)$ which acts on the points of the affine line over the finite field $\mathbb{F}_m$ as a sharply doubly transitive permutation group. For $m$ prime, there exists no other examples. For $m=r^2$ with $r>2$ prime there exists further examples arising from nearfields of degree $r^2$.
\end{result}
The group ${\rm{A\gamma L}}(2,r^2)$ arises from the regular nearfield of degree $r^2$ and consists of all permutations on the elements of the finite field $\mathbb{F}_{r^2}$ which are of the form $x\mapsto a\circ x+b$ where $a,b\in \mathbb{F}_{r^2}$ and $a\circ x=ax$ for $a$ square in $\mathbb{F}_{r^2}$ while $a\circ x=ax^r$ for non-square $a$ in $\mathbb{F}_{r^2}$. For $r\in \{5,7,11,23,20,59\}$, there exist irregular nearfields each of them gives rise to a sharply doubly transitive group as the regular nearfield does; see \cite[Section 9]{huppert3}.
For smaller values of $m$, the following holds. For $m=9$ there exist exactly two sharply doubly transitive permutation groups, namely $\rm{AGL}(1,9)$ and $\rm{A\gamma L}(1,9)$, whereas for $m=25$ three, namely $\rm{AGL}(1,25)$, $\rm{A\gamma L}(1,25)$, and ${\mathcal{N}}(5)\cong (C_5\times C_5)\rtimes \rm{SL}(2,3)$
arising from the unique irregular nearfield of degree $25$. In particular, ${\rm{A\gamma L}}(1,9)\cong \PSU(3,2)$. Furthermore, the $1$-point stabilizer of $\rm{A\gamma L}(1,25)$ contains a subgroup of order $12$ while that of ${\mathcal{N}}(5)$, isomorphic to $\rm{SL}(2,3)$, does not.

\section{Doubly transitive groups on curves with simple minimal normal subgroup}
\begin{theorem}
\label{doubly} Let $G$ be a group acting on a finite set $\Omega$ with $|\Omega|>2$ such that
\begin{itemize}
\item[\rm(i)] $G$ acts on $\Omega$ as a $2$-transitive permutation group,
\item[\rm(ii)] the action of $G$ on $\Omega$ is faithful,
\item[\rm(iii)] the $1$-point stabilizer has a normal Sylow $p$-subgroup with cyclic complement.  
\end{itemize}
If $G$ has a simple non-abelian normal minimal subgroup $W$ then either $G\cong {\rm{P\Gamma L}}(2,8)$ and
$W\cong \PSL(2,8)$ with $|\Omega|=28$ and $p=3$, 
or one of the following cases occurs: $W\cong \PSL(2,q), Sz(q), \PSU(3,q),$ $ Ree(q)$, where $q$ is a power of $p$, and
$|\Omega|$ is $q+1$ in the linear case, $q^{2}+1$ in the Suzuki case, $q^{3}+1$ in the unitary and Ree case.
\end{theorem}
\begin{proof}  The $1$-point stabilizer of $G$ is solvable. In particular, since $G$ is not solvable, it does not contain any regular normal subgroup.

First the case where $\Omega$ has odd size is investigated. As a minimal normal subgroup of a solvable group is abelian, Result \ref{onanre} applies. In case (i) of Result \ref{onanre}, since  $W$ acts on $\Omega$ as $\PSL(r+1,q)$ on the points of the projective space $\PG(r,q)$, the $1$-point stabilizer of $\PSL(r+1,q)$ contains the linear group ${\rm{SL}}(r,q)$ which is solvable only when either $r=1$, or $r=2$ and
 $q=2,3$. If $r=1$, (iii) of Result \ref{res74} yields $p=2$ as the unique maximal normal subgroup of the $1$-point stabilizer has order $q$, and $q+1$ is odd.
If $r=q=2$ then $|\Omega|=7$ and hence the $1$-point stabilizer is isomorphic to $\bf{S}_4$, but the Sylow $2$-subgroup of ${\bf{S}}_4$ is not a normal subgroup of ${\bf{S}}_4$, and  Condition (iii)  yields that this case cannot actually occur. If $r=2,q=3$ then $|\Omega|=13$ and the $1$-point stabilizer contains a subgroup isomorphic to ${\rm{SL}}(2,3)$ that contains no normal $3$-subgroup. But,  by Condition (iii), this is impossible.


If the size of $\Omega$ is even, Result \ref{holtre} applies. Apart from the exceptional cases, the $1$-point stabilizer has a unique normal subgroup of order $q$, and hence Condition (iii) yields that $q$ is a power of $p$. If $G\cong {\rm{P}}\Gamma L(2,8)$, $W\cong \PSL(2,8)$ and $|\Omega|=28$ then the $1$-point stabilizer contains a non-cyclic normal subgroup of order $27$, and  Condition (iii) yields $p=3$.
\end{proof}

\begin{proposition}
\label{pro3jan} Let $G$ be a group acting on a finite set $\Omega$ with $|\Omega|>2$ such that Conditions (i), (ii) and (iii)  of Theorem \ref{doubly} are satisfied. Assume that $G$ has a simple non-abelian minimal normal subgroup $W$. If the $1$-point stabilizer $T$ of $G$ has a subgroup $H$ of order $|\Omega|-1$ that acts (sharply) transitively on the remaining $|\Omega|-1$ points then $H$ is a normal subgroup of $T$.
\end{proposition}
\begin{proof} 
Theorem \ref{doubly} applies.

If $G\cong {\rm{P\Gamma L}}(2,8)$ and $|\Omega|=28$ then the $1$-point stabilizer $T$ has order $54$ and contains only one subgroup of order $27$.  Hence, the latter one is $H$, and it is normal in $T$.



If $W\cong\PSL(2,q)$ with $q\geq 4$ (and $|\Omega|=q+1$ with $q=p^h$) then $\PSL(2,q)\leq G \leq {\rm{P\Gamma L}}(2,q)$.
Assume that $H$ is not contained in $\PSL(2,q)$, and look at the subgroup $L$ generated by $\PSL(2,q)$ and $H$. The subgroup $\PSL(2,q)\cap H$ is a $p$-subgroup of $\PSL(2,q)$ which fixes $P$. The stabilizer $W_P$ of  $P$ in $W$ has a Sylow $p$-subgroup $R$ of $\PSL(2,q)$, and $\PSL(2,q)\cap H$ is contained in $R$. Since $W_P$ is a normal subgroup of $G_P$, $RH$ is a $p$-subgroup of $L$ whose order equals $|R||H|/|R\cap H|$. Thus, $RH$ is a Sylow $p$-subgroup of $L$. From Condition (iii) of Theorem \ref{doubly} applied to $L_P$, $RH$ is a normal subgroup of $L_P$. From Lemma \ref{lem11jan}, $R=H$.

If $W\cong \PSU(3,q)$ with $q\geq 3$ (and $|\Omega|=q^3+1$ with $q=p^h$) then $\PSU(3,q)\leq G \leq {\rm{P\Gamma U}}(3,q)$. The above argument used for $\PSL(2,q)$ still works with $|H|=q^3$ and Lemma \ref{lem11Ajan}.

If $W\cong Sz(q)$ (and $|\Omega|=q^2+1$ with $q=2^h, h\ge 3$ odd) then $|H|=2^{2h}$ but $[\aut(Sz(q)):Sz(q)]=h$ is odd. Therefore, up to conjugacy, $H= Sz(q)$. The $1$-point stabilizer of $Sz(q)$ has a unique (Sylow) $2$-subgroup of order $q^2$ which acts transitively on the set of the remaining $|\Omega|-1$ points. In particular, that Sylow $2$-subgroup is normal and coincides with $H$.

If $W\cong Ree(q)$ (and $|\Omega|=q^3+1$ with $q=3^h, h\ge 1$ odd) then $|H|=3^{3h}$ and $[\aut(Ree(q)):Ree(q)]=h$. The above argument used for $\PSL(2,q)$ still works with $|H|=q^3$ and Lemma \ref{lem11Bjan}.
\end{proof}
\begin{rem}
\label{remdoubly}{\em{ By (iii) of Result \ref{res74},
 both Theorem \ref{doubly} and Proposition \ref{pro3jan} are valid for  $\aut(\cX)$ provided that  Conditions (i) and (ii) in Theorem  \ref{doubly} are satisfied.}}
 \end{rem}
\section{Doubly transitive groups on curves with solvable minimal normal subgroup}
\begin{theorem}
\label{doublysolv} Let $G$ be a subgroup of $\aut(\cX)$ which has an orbit $\Omega$ with $|\Omega|>2$ such that both (i) and (ii) in Theorem \ref{doubly} hold. If, in addition,
\begin{itemize}
\item[\rm(iii)] $G$ has a solvable minimal normal subgroup $N$,
\item[\rm(iv)] the $1$-point stabilizer of $G$ has a subgroup $T$ that is sharply transitive on the remaining points of $\Omega,$
\item[\rm(v)] the quotient curve $\cX/T$ is rational,
\end{itemize}
then $\cX$ is either rational, or elliptic.
\end{theorem}
\begin{proof} Let $d=|\Omega|$. Since $N$ is faithful and sharply transitive on $\Omega$, $T\cap N$ is trivial, the subgroup $S=TN$ has order $d(d-1)$ and hence it is a sharply doubly transitive group on $\Omega$. Therefore, $S$ has a partition whose components are the subgroup $N$ of order $d$ together with the stabilizers $S_U$ in $S$ with $U$ ranging over $\Omega$. Result \ref{th11.114} applies to $S$ with $k=1+d$, where $S_1=N$, and, for $i=2,\ldots k$, $S_i$ are the conjugates of $T$ in $S$. In particular, the quotient curves $\cX/S_i$ for $i\geq 2$ are isomorphic. Since one of them, namely $\cX/T$ is rational, we have $\gg(\cX/S_i)=0$ for $i=2,\ldots k$. Also $\gg(\cX/S)=0$, as $T$ is a subgroup of $S$. Now, (\ref{part}) reads $m\gg(\cX)=m\gg(\cX/N)$ whence $\gg(\cX)=\gg(\cX/N)$. This is only possible when either $\gg(\cX)=0$ or $\gg(\cX)=1$. \end{proof}
\begin{rem} {\em{Theorem \ref{doublysolv} is special case of a more general result of Guralnick; see \cite[Corollary 3.2]{gural}.}}
\end{rem}
\begin{proposition}
\label{pro23jan} Let $\cX$ be a rational curve. If $G$ is a subgroup of $\aut(\cX)$ such that both (i) and (ii) in Theorem \ref{doubly} hold, and, in addition, $G$ has a solvable minimal normal subgroup then one of the following cases occurs.
\begin{itemize}
\item[(i)] $G$ is sharply doubly transitive on $\Omega$, $G\cong \AGL(1,m)$ with $|\Omega|=m$ where either $m$ is a power of $p$, or $m=3$ and $p\neq 3$, or $m=4$ and  $p\neq 2$.
\item[(ii)] $|\Omega|=4$, $G\cong \bS_4$, $p\neq 2$, and  $\AGL(1,4)\cong\bA_4$ is the unique subgroup of $G$ which is sharply doubly transitive on $\Omega$.
\end{itemize}
\end{proposition}
\begin{proof} From the proof of Theorem \ref{doublysolv}, $S=TN$ is a sharply doubly transitive group on $\Omega$.
In particular, the order of $S$ is the product of two consecutive integers. From Result \ref{lem30oct2016} applied to $S$, we have $S\cong \AGL(1,m)$ where either $m$ is a power of $p$, or
$m=3$ and $p\neq 3$, or $m=4$ and  $p\neq 2$. Moreover, if $m$ is a power of $p$ then any solvable subgroup of $\PGL(2,\mathbb{K})$ containing $\rm{AGL}(1,m)$ has an abelian subgroup of order $m'=mp^r$ with $r>1$. Therefore, $G$ cannot contain $S$ properly. Also, $\rm{AGL}(1,3)$ is the only doubly transitive permutation group of degree $3$, and hence $G=S$ for $m=3$ and $p\neq 3$. Finally, there are two doubly transitive permutation groups of degree $4$, one is $\rm{AGL}(1,4)\cong \bA_4$ the other $\bS_4$, and in the former case $G=S$ but $[G:S]=2$ in the latter.
\end{proof}
\begin{proposition}
\label{pro23janA}
Let $\cE$ be an elliptic curve. If $G$ is a subgroup of $\aut(\cE)$ such that both (i) and (ii) in Theorem \ref{doubly} hold
 then one of the following occurs.
\begin{itemize}
\item[(i)] $G$ is sharply doubly transitive on $\Omega$, $G\cong \AGL(1,m)$ with $m=|\Omega|$ where $m=3,4,5,7$ for $p\neq 2,3$, and $m=3,4,5,7$ for $p=3$, and $m=3,5,7$ for $p=2$,
\item[(ii)] $G$ is sharply doubly transitive on $\Omega$, $G\cong \PSU(3,2)$ where $|\Omega|=9$ and $p=2$.
\item[(iii)] $G$ is sharply doubly transitive on $\Omega$, $G\cong (C_5\times C_5)\rtimes\rm{SL}(2,3)$ where $|\Omega|=25$ and $p=2$.
\item[(iv)] $G$ is not sharply doubly transitive on $\Omega$, $G\cong \bS_4$ where $|\Omega|=4$, $p\neq 2$,.
\item[(v)] $G$ is not sharply doubly transitive on $\Omega$, $G\cong {\rm{A}}\Gamma L(1,9)$ where $|\Omega|=9$ and $p=2$.
\end{itemize}
\begin{proof} Since a $1$-point stabilizer $G_P$ of $G$ has order at least $|\Omega|-1$, Result \ref{autelliptic} gives the possibilities for $|\Omega|$, namely  $|\Omega|=3,5,7$ for $p\neq 2,3$,  and
$|\Omega|=3,5,7,13$ for $p=3$, and $3,4,5,9,25$ for $p=2$. Comparison of the cases listed in (i),\ldots,(v) with Result \ref{sharply} (and the subsequent remark) shows that only two cases have to be ruled out, namely $|\Omega|=13$ for $p=3$, and $|\Omega|=4$ for $p=2$.
In the former case, $G$ is sharply doubly transitive, and since $13$ is a prime $G\cong \AGL(1,13)$ and its $1$-point stabilizer $G_P$ is cyclic; see Result \ref{sharply}. On the other hand,
$G_P$ is not abelian in this case by Result \ref{silv}, a contradiction. In the latter case, $G\cong \AGL(1,4)$, and  $p=2$.
Since $j(\cE)=0$, $\cE$ has zero $2$-rank and hence it has no translation of order $2$. On the other hand the only non-trivial normal subgroup of $\AGL(1,4)$ has order $4$. But this contradicts
Result \ref{res12feb2019}. This contradiction ends the proof.
\end{proof}
\begin{proposition}
\label{pro24jan} Let $G$ be a subgroup of $\aut(\cX)$ which has an orbit $\Omega$ such that both (i) and (ii) in Theorem \ref{doubly} hold. If, in addition,
\begin{itemize}
\item[\rm(iii)] $G$ has a solvable minimal normal subgroup $N$,
\item[\rm(iv)] the $1$-point stabilizer of $G$ has a subgroup $T$ that is sharply transitive on the remaining points of $\Omega,$
\item[\rm(v)] the quotient curve $\cX/T$ is rational,
\end{itemize}
then $T$ is a normal subgroup of the $1$-point stabilizer of $G$.
\end{proposition}
\begin{proof} In Propositions \ref{pro23jan} and \ref{pro23janA}, either $T$ coincides with the $1$-point stabilizer of $G$, or $T$ is an index $2$ subgroup of it.
\end{proof}
\end{proposition}
\section{Auxiliary results for the proof of Theorem \ref{th1}}
In this section, $P_1,P_2$ are distinct points of $\cX$, and $G_1,G_2$ are distinct subgroups of $\aut(\cX)$ where $G_1$ fixes $P_1$ and $G_2$ fixes $P_2$. Moreover, $|\supp(D)|>2$, and $G_1,G_2$ have properties (I),(II),(III). By Lemmas \ref{profuka} and \ref{profuka1}, properties (i) ,(ii) and (vi) of Lemma \ref{profuka1} also hold.

As before, let $\Omega$ denote $\supp(D)$ of the divisor $D$ of $\cX$ defined in (III). Then (ii) of Lemma \ref{profuka1} states that $G$ acts on $\Omega$ as a doubly transitive permutation group.
Actually, the normal closure $S$ of $G_1$ in $G$ still acts doubly transitively on $\Omega$. In fact, there exists $g\in G$ which takes $P_1$ to $P_2$ and the subgroup $H_2=g^{-1}G_1g$ of $G$ fixes $P_2$ and acts (sharply) transitively on $\Omega\setminus \{P_2\}$. Hence $G_1,H_2$ also have properties (I),(II),(III).

 Our aim is to determine all possibilities for $S$.  Since $S$ may happen to be not faithful on $\Omega$, we begin by investigating the subgroup $K$ of  $G$ consisting of all elements which fix $\Omega$ pointwise.

\begin{lemma}
\label{prop1jan}  $K$ is a cyclic group whose order is prime to $p$ and divides $\deg(\cC)$. Furthermore,  $K=Z(G)=Z(S)$. 
\end{lemma}
\begin{proof}
From (vi) of Lemma \ref{profuka1}, the poles of $x$ are the points of $\Omega$ different from $P_1$, each with multiplicity $1$. Take a non-trivial element  $\alpha\in K$ of order $s$. For any $v\in \mathbb{K}(\cC)$, $\alpha$ takes a pole of $v$ with multiplicity $m$ to a pole of $\alpha(v)$ with the same multiplicity $m$. Therefore, $\alpha(x)$ has the same poles of $x$.

We show that $p$ does not divide $|K|$. By way of a contradiction, assume $s=p$. From Lemma \ref{lemA1jan}, no point $P\in \Omega$ is a pole of $\alpha(x)-x$. Also, no branch of $\cC$ centered at an affine point is a pole of $\alpha(x)-x$. Thus $\alpha(x)-x\in \mathbb{K}$. Similarly, $\alpha(y)-y \in\mathbb{K}$. Therefore, $\alpha$ is a translation, that is, $\alpha(x)=x+a, \alpha(y)=y+b$ for $a,b\in \mathbb{K}$, and it has order $p$. Assume that $\alpha\beta\neq \beta\alpha$ for some $\alpha\in K$ and $\beta\in G_1$. Then $\beta^{-1}\alpha\beta(x)=\beta^{-1}(\alpha(x))=\beta^{-1}(x+a)=\beta^{-1}(x)+\beta^{-1}(a)=x+a.$ Therefore $\alpha^{-1}\beta^{-1}\alpha\beta(x)=x$. Since $\mathbb{K}(x)=\cX^{G_1}$
this yields $\alpha^{-1}\beta^{-1}\alpha\beta\in G_1$. On the other hand $\alpha^{-1}\beta^{-1}\alpha\beta$ fixes $\Omega$ pointwise. Therefore  $\alpha^{-1}\beta^{-1}\alpha\beta$ is the identity but this contradicts $\alpha\beta\neq \beta\alpha$.  Therefore $\alpha$ centralizes $G_1$. As the same holds for $G_2$, $\alpha\in Z(G)$ follows.
 For a translation $\alpha\in K$, let $T$ denote its center. Take a point $P\in \supp(D)$ such that $\varphi(P)$ is different from $T$. Let $\gamma$ be the branch of $\cC$ associated with $P$. Then $\gamma$ is centered at $\varphi(P)$, and its tangent $t$ is different from the line at infinity by (vi) of Lemma \ref{profuka1}. Then $\alpha$ does not leave invariant $t$ and hence $\alpha$ does not fix $P$, a contradiction which shows that $K$ contains no translation. Therefore, $p\nmid |K|$.

 For $p\nmid s$, the same argument may be used. In fact, Lemma \ref{lemA1jan} shows that no point $P\in \Omega$ is a pole of $\alpha(x)-ux$ where $u$ is a non-trivial $m$-th root of unity and $m$ is the smallest integer for which $\alpha^m(x)=x$. Thus $\alpha(x)=ux+b$ with $b\in \mathbb{K}$, and similarly $\alpha(y)=ry+c$ with some $r\in \mathbb{K}$. Since $\alpha$ fixes a point $\varphi(Q)\in \ell_\infty$ other than $\varphi(P_1)$ and $\varphi(P_2)$, $\alpha$ is a homology. Therefore $u=r$ and the center of $\alpha$ is in the point $(-b/(u-1),-c/(u-1))$. From this, $\alpha\beta=\beta\alpha$,  and hence $K\leq Z(G)$ follows.
 As before, for a point $P\in \supp(D)$, let $\gamma$ be the branch of $\cC$ associated with $P$, centered at $\varphi(P)$, and  with tangent $t$ different from the line at infinity. Then the homology $\alpha$ leaves $t$ invariant, and hence $t$ passes through the center of $\alpha$. This shows that the tangents to the branches of $\cC$ arising from the points in $\supp(D)$ are concurrent at the center of $\alpha$. Furthermore,
 since any group generated by two homologies with different centers contains a translation, it turns out that $K$ consists of homologies with the same center $C$. In particular, $K$ is isomorphic to a finite multiplicative subgroup of $\mathbb{K}$. Therefore, $K$ is cyclic and $p\nmid |K|$. Since $G_1$ fixes $\varphi(P_1)=Y_\infty$ and $Y_\infty$ is a simple point of $\cC$, the tangent to $\cC$ at $Y_\infty$ contains no point of $\cC$ other than $Y_\infty$. Therefore $C$ is not a point of $\cC$. Take a line $\ell$ through $C$ and disjoint from $\Omega$ such that $\ell$ intersects $\cC$ in non-singular points. From every $K$-orbit $\Delta_j$ in $\ell\cap \cC$, take a unique point $R_j$. Then for the intersection divisor $\cC\circ \ell$,  B\'ezout's theorem gives $\deg(\cC)=\deg(\cC\circ \ell)=\sum_j |\Delta_j| I(R_j,\cC\cap \ell)$. Also, $|\Delta_j|=|K|$ as no non-trivial element in $K$ fixes a point in $\ell\cap \cC$. From this $|K|$ divides $\deg(\cC)$.

 Finally, since any point in $\Omega$ is the only fixed point of a conjugate of $G_1$ in $G$, $Z(S)$ fixes $\Omega$ pointwise. Therefore $Z(G)\le Z(S)\le K \le Z(G)$ whence $K=Z(G)=Z(S)$. 
\end{proof}

A useful ingredient in the proof of Theorem \ref{th1} is the following result.
\begin{theorem}
\label{cond*} $G_1$ is a normal subgroup of the stabilizer of $P_1$ in $G$.
\end{theorem}
 \begin{proof} By Propositions \ref{pro3jan} and \ref{pro24jan}, $K$ may be assumed to be non-trivial. Let $\bar{G}$ be the doubly transitive permutation group induced by $G$ on $\Omega$. Then  $\bar{G}$ acts on $\Omega$ as $G$ does, and no nontrivial element in $\bar{G}$ fixes $\Omega$ pointwise.
 Propositions \ref{pro3jan} and \ref{pro24jan} apply to the quotient curve $\bar{\cX}=\cX/K$. Therefore, $\bar{G_1}=G_1K/K$ is a normal subgroup of the stabilizer of $\bar{P}_1$ in $\bar{G}$ where $\bar{P_1}$ is the point lying under $P_1$ in the cover $\cX|\bar{\cX}$. Therefore, $G_1K$ is a normal subgroup of the stabilizer of $P_1$ in $G$. From Proposition \ref{prop1jan}, $|K|$ divides $\deg(\cC)=|\Omega|$ and $K=Z(G)$ whereas $|G_1|=|\Omega|-1$ by (iii) of Lemma \ref{profuka1}. Thus $G_1K=G_1\times K$ with $\rm{g.c.d.}(|G_1|,|K|)=1$. Therefore, $G_1$ is a characteristic subgroup of $G_1\times K$, and hence $G_1$ is a normal subgroup of $G_{P_1}$.
\end{proof}

\begin{rem} An alternative proof for Theorem \ref{cond*} can be carried out by using Results \ref{holtre} and \ref{onanre}.
\end{rem}

\section{Proof of Theorem \ref{th1}}
Let $\ell$ denote the line through $\varphi(P_1)$ and $\varphi(P_2)$.

The case where $\varphi(P)$ with $P\in \cX$ lies on $\ell$ only for $P=P_1$ or $P=P_2$ cannot occur since in this case (\ref{eq14mar}) does not hold and hence at least one of the points $\varphi(P_1)$ and $\varphi(P_2)$ of $\cC$ is singular.

From now on we assume that $\varphi(P)\in \ell$ for some $P\in \cX$ other than $P_1$ and $P_2$. Then $|\Omega|>2$ where $\Omega=\supp(D)$.
Theorem \ref{cond*} allows us to apply Result \ref{HC} to the group space $(\Omega,G)$ with $Q=G_1$ where $S$ is the normal closure of $G_1$ in $G$.
\subsection{$S$ is of type (i) in Result \ref{HC}} $S$ is simple for $S=\PSL(2,q), q>3,Sz(q),\PSU(3,q),q>2,Ree(q)$ and Theorem \ref{cond*} applies showing that $q$ is a power of $p$. In the other non-solvable case we have either $S={\rm{SL}}(2,q), q>3$ or  ${\rm{SU}}(3,q), q>2$, and $S$ acts on $\Omega$ as $\PSL(2,q)$, or $\PSU(3,q)$ in their natural $2$-transitive representation. This permutation representation has non-trivial kernel $z$. Thus Theorem \ref{cond*} applies to the quotient curve $\cX/Z$, and it shows that $q$ is a power of $p$.
In the remaining cases, $S$ is one of the solvable groups $\PSL(2,2),\PSL(2,3),{\rm{SL}}(2,2),{\rm{SL}}(2,3),\PSU(3,2),{\rm{SU}}(3,2)$. If either $S=\PSL(2,2)\cong \AGL(1,3)$, or $S=\PSL(2,3)\cong \AGL(1,4)$, or $S=\PSU(3,2)$, the permutation representation of $S$ on $\Omega$ is faithful and sharply doubly transitive. These cases are also of type (iii) in Result \ref{HC} and are treated below; see Subsection \ref{ss1}. Also, $S={\rm{SU}}(3,2)$ falls in case (iv) of Result \ref{HC} and it is investigated later; see Subsection \ref{ssiv}.

We are left with the case $S={\rm{SL}}(2,3)$ (and $|\Omega|=4$). From (iii) of Lemma \ref{profuka1}, $\deg(C)=4$, and hence $\gg(\cX)\leq 3$. We show that $\gg(\cX)=3$.

Since ${\rm{SL}}(2,3)$ is not a subgroup of $\PGL(2,\mathbb{K})$ by Result \ref{lem30oct2016}, $\cX$ is not rational.

Assume that $\cX$ is an elliptic curve $\cE$. We show that $|G\cap J(\cE)|=4$. For any point $Q\in \Omega$, there exists $h\in G$ which takes $P_1$ to $Q$. On the other hand, $J(\cE)$ has a translation $\tau$ taking $Q$ to $P_1$. Then $\tau h$ fixes $P_1$. Since the stabilizer of $P_1$ in $\aut(\cE)$ has order $\le 6$ whereas the stabilizer of $P_1$ in $S={\rm{SL}}(2,3)$ has order $6$, it turns out that every automorphism in $\aut(\cE)$ fixing $P_1$ is in $S$. Therefore, $\tau h$ and hence $\tau$ itself is in $S$ whence $|S\cap J(\cE)|\geq 4$ follows. As
no non-trivial translation fixes a point of $\cE$, this yields $|S\cap J(\cE)|=4$. From  Result \ref{res12feb2019},
$S\cap J(\cE)$ is a normal subgroup of $S$. This contradicts the fact that ${\rm{SL}}(2,3)$ has no normal subgroup of order $4$.

Assume that $\gg(\cX)=2$. The Hurwitz genus formula applied to $G$ yields that $G$ has exactly three short orbits, of length $4,6$ and $12$, respectively. In particular, each point in the orbit of length $12$ is fixed by an involution. Since ${\rm{SL}}(2,3)$ has only one involution $h$, this yields that $h$ has at least $12$ fixed points. This contradicts the fact that no non-trivial automorphism of a genus $\gg$ curve may have more than $2\gg+2$ fixed points; see \cite[Lemma 11.12]{HKT}.

Therefore, $\gg(\cX)=3$ and hence it $\cC$ is a non-singular curve of degree four.

All cases  occur as shown by the examples exhibited in Section \ref{esempi}. Here we observe that $\gg(\cX)\geq 2$ apart from the possibilities where $S\cong \PSL(2,q)$ and $|\Omega|=q+1$, or $S\cong \bA_5$ and $|\Omega|=5$. This follows by comparison of the list in (i) of Result \ref{HC} with Result \ref{lem30oct2016} (for $\gg(\cX)=0$) and with Result \ref{autelliptic} (for $\gg(\cX)=1$).
\subsection{$S$ is of type (ii) in Result \ref{HC}} An example is the smallest Ree curve; see Section \ref{esempi}.
\subsection{$S$ is of type (iii) in Result  \ref{HC}}
\label{ss1} Proposition \ref{pro23jan} applies to $S$, and the possibilities come from Propositions \ref{pro23jan} and \ref{pro23janA}. All cases occur; see Section \ref{esempi}.
\subsection{$S$ is of type (iv) in Result \ref{HC}}
\label{ssiv}
Our goal is to show that $S\cong {\rm{SU}}(3,2)$ and $\gg(\cX)=10$. In case (iv) of Result \ref{HC}, $|Z(S)|=d$ with $|\Omega|=d^2$. Furthermore,
the quotient curve $\tilde\cX=\cX/G_1$ is rational and the quotient group $\tilde{Z}=(Z(S)\times G_1)/G_1$ is a subgroup of $\aut(\tilde{\cX})$ isomorphic to $Z(S)$. Since $Z(S)$ fixes $\Omega$ pointwise whereas $G_1$ has two orbits on $\Omega$, we have that $\tilde{Z}$ has at least two fixed points in $\tilde{\cX}$.
Therefore, $p$ is prime to the order of $\tilde{Z}$, that is, $p\neq d$. Also, $\tilde{Z}$ has no further fixed point. This shows that $\Omega$ coincides with the set of all fixed points of $Z(S)$. Now, look at the quotient curve $\bar{\cX}=\cX/Z(S)$. From the Hurwitz genus formula, $2\gg(\cX)-2=d(2\gg(\bar{\cX})-2)+d^2(d-1)$. Since $\bar{S}=S/Z(S)$ is sharply doubly transitive on $\Omega$, Theorem  \ref{doublysolv} applies to $\bar{\cX}$. Thus, $\bar{\cX}$ is either rational, or elliptic. In the former case, as $d\neq p$, Result \ref{lem30oct2016} yields $\bar{S}\cong {\bf{A}}_4$. This implies $d=2$, a contradiction.

 Therefore, $\bar{\cX}$ is elliptic, and $\gg(\cX)=\ha(d^2(d-1)+2)$. Also, the quotient group $\bar{G}_1=(G_1\times Z(S))/Z(S)$ is a subgroup of $\aut(\bar{\cX})$ fixing the point $\bar{P}_1$ of $\bar{\cX}$ lying under $P_1$ in the cover $\cX|\bar{\cX}$. Since $\bar{G}_1\cong G_1$ and $|G_1|=d^2-1$ with $d\ge 3$, Result \ref{autelliptic} yields $p=2$ and $d=3,5$. For $d=3$, we have $|S|=216$. More precisely, a MAGMA computation shows that either $S\cong {\rm{SU}}(3,2)=SmallGroup(216,88)$, or $S\cong SmallGroup(216,160)$. The latter case cannot actually occur since the $3$-Sylow subgroup of $SmallGroup(216,160)$ is abelian, and hence is not extra-special.

 We are left with the possibility that $p=2$, $d=5$, $\gg(\cX)=51$, and $|S|=3000$. Since $16\nmid 3000$, a Sylow $2$-subgroup $S_2$ of $G_1$ is also a Sylow $2$-subgroup of $S$. Obviously, $S_2$ fixes $P_1$. We show that no non-trivial element in $S_2$ fixes a point other than $P_1$. The quotient group $\bar{S}_2=(Z(S)\times S_2)/Z(S)$ is isomorphic to $S_2$ and it is a subgroup of $\aut(\bar{\cX})$ which fixes $\bar{P}_1$. From Result \ref{autelliptic}, $\bar{S}_2$ (and hence $S_2$) is isomorphic to the quaternion group $Q_8$ of order $8$. The quotient curve $\hat{\cX}=\bar{\cX}/\bar{S}_2$
is rational, and it has zero $2$-rank. From Result \ref{lem11.131HKT}, $\bar{\cX}$ has also zero $2$-rank. Therefore,
no non-trivial element in $\bar{S}_2$ fixes a point of $\cX$ other than $\bar{P}_1$. This yields that $S_2$ fixes $P_1$ but its non-trivial elements fix no point other than $P_1$. To apply the Hurwitz genus formula to $S_2$, compute the ramification groups of $S_2$ at $P_1$. By definition, $S_2=S_2^{(0)}=S_2^{(1)}$. From Result \ref{lem11.75(i)} applied to a generator $\alpha$ of $Z(S)$, we have $S_2^{(1)}=\ldots =S_2^{(5)}$. Since $S_2$ is not an elementary abelian group, (ii) of Result \ref{res74} yields that $S^{(6)}$ is non-trivial. Therefore, $S^{(6)}$ contains the (unique) subgroup $T$ of $S_2$ of order $2$. Since $T$ is in $G_1$ and $G_1$ contains a (cyclic) subgroup $C_{15}$ of order $15$, Result \ref{lem11.75(i)} applies to a generator $\alpha$ of $C_{15}$ whence $S_2^{(i)}$ for contains $T$ for $i=6,\ldots 15$. Let $\cX'=\cX/S_2$. From the Hurwitz genus formula applied to $S_2$,
\begin{equation}
\label{eq13dic2018}
100=2(\gg(\cX)-1)\ge 16(\gg(\cX')-1)+42+10
\end{equation}
whence $\gg(\cX')\le 4$. Moreover, $(C_{15}\times S_2)/S_2\cong C_{15}$ is a subgroup of $\aut(\cX')$ which fixes
the point $P_1'$ lying under $P_1$ in the cover $\cX|\cX'$.

If $\cX'$ is rational, then the subgroup $(Z(S)\times S_2)/S_2\cong C_5$ of $\aut(\cX')$ fixes exactly two points, namely $P_1'$ and $U'$. Therefore, the fixed points of $C_5$ are $P_1$ and some (or all) of the points in the $S_2$-orbit lying over $U'$. This shows that $C_5$ has at most $9<25$ fixed points, a contradiction.

We may assume that $\gg(\cX')\ge 1$. Result \ref{theorem11.60HKT} yields $15\le 4\gg(\cX')+2$ whence $\gg(\cX')=4$. This shows that equality holds in (\ref{eq13dic2018}). In particular, $S_2=S_2^{(i)}$, for $i=0,1,\ldots,5$, and $T=S_2^{(i)}$ for
$i=6,\ldots 15$, and $S_2^{(16)}=\{1\}$. From (\ref{eq1bis}) applied to $G_1$, we have then $d_{P_1}=23+5\cdot 7+
10=68$. Let $C_3$ be the subgroup of $C_{15}$ of order $3$. Then the quotient group $C_3'=(S_2\rtimes C_3)/S_2\cong C_3$ is a subgroup of $\aut(\cX')$. Let $\check{\cX}$ be the quotient curve $\cX'/C_3'$. The Hurwitz genus formula applied to $C_3'$ reads $6=2(\gg(\cX')-1)=6(\gg(\check{\cX})-1)+2r$
where $r$ counts the fixed points of $C_3'$. Here $r\ge 1$ as $C_3'$ fixes $P_1'$. From this, $\gg(\check{\cX})\le 1$, and $r=3$ or $r=6$ according as $\check{\cX}$ is elliptic or rational. The former case cannot actually occur by Result \ref{autelliptic}, since
$(Z(S)\times (S_2\rtimes C_3))/(S_2\rtimes C_3)\cong C_5$ is a subgroup of $\aut(\check{\cX})$ fixing the point lying under the point $P_1$ in the cover $\cX|\check{\cX}$. Therefore, $\check{\cX}$ is rational, and $r=6$.
Take a fixed point $U'$ of $C_3'$ other than
$P_1'$ and consider the $S_2$-orbit $\Delta$ lying over $U'$. Since $C_3$ leaves $\Delta$ invariant, and $|\Delta|=8$, $C_3$ has at least two fixed points in $\Delta$. Therefore, $C_3$ has at least $12$ fixed points. Moreover, $G_1$ has four (pairwise conjugate) subgroups of order $3$. Now, the Hurwitz genus formula applied to $G_1$ reads,
$100=2(\gg(\cX)-1)\ge -48+68+4\cdot24=116$ a contradiction.

\subsection{$S$ coincides with $G$} By way of a contradiction, assume that some non-trivial element $g\in G_2$ does not belong to $S$. Since $S$ is a normal subgroup of $G$, $g$ is in the normalizer of $Z(S)$. Let $\bar{S}=S/Z(S)$ and $\bar{g}=gZ(S)/Z(S)$. We show that $\bar{g}\not\in Z(\bar{S})$. Assume on the contrary that $gsg^{-1}s^{-1}\in Z(S)$ for every $s\in S$. Since $Z(S)$ fixes $\Omega$ pointwise,  this yields $gs(P_2)=sg(P_2)=s(P_2)$. As $P_2$ is the unique fixed point of $g$, it follows $s(P_2)=P_2$, a contradiction $S$ being  transitive on $\Omega$. Therefore, $\bar{g}$ induces by conjugation a non-trivial automorphism of $\bar{S}$.

If $S$ is of type (i) in Result \ref{HC} then  $d-1$ a power of $p$ and $\bar{S}$ is isomorphic to one of the groups $L=\PSL(2,q),\, \PSU(3,q),\,Sz(q),\,Ree(q)$, and the action of $\bar{S}$ on $\Omega$ is the natural doubly transitive permutation representation of $L$. If $L=\PSL(2,q)$ then $L$ together with $\bar{g}$ generate a subgroup $D$ of ${\rm{P}}\Gamma L(2,q)$ strictly containing $\PSL(2,q)$. From (iii) of Result \ref{res74}, the stabilizer $M$ of $P_2$ in $D$ is the semidirect product of the Sylow $q$-subgroup of $D$ fixing $P_2$ by a cyclic complement. Now the second claim in Lemma \ref{lem11jan} yields that $D\leqq L$, a contradiction. Similar arguments can be used to investigate the other possibilities for $L$ where Lemma \ref{lem11jan} by Lemmas \ref{lem11Ajan}, \ref{lem11Cjan}, and \ref{lem11Bjan}, respectively.

If $S$ is of type (ii) in Result \ref{HC} then $S\cong {\rm{P\Gamma L}}(2,8)=\aut({\rm{P\Gamma L}}(2,8))\cong \aut(S)$, and hence $\bar{g}\in\bar{S}$, a contradiction.

If $S$ is of type (iii) in Result \ref{HC} then $\cX$ is either rational, or elliptic and one of the cases in Propositions \ref{pro23jan} and \ref{pro23janA} occurs. Let $N$ be the (unique) minimal normal subgroup of $S$. Then $N$ is a characteristic subgroup of $S$, and hence it is a minimal normal subgroup of $G$. Furthermore, $G_1N\le S$ is a sharply doubly transitive group on $\Omega$. Thus $S=G_1N$. Since $S\leq G$, either $G=S$, or $G>S$ and Lemma \ref{pro23janA} shows that $G_1N$ is the unique sharply doubly transitive subgroup of $G$ on $\Omega$. Since $G_2N$ is another sharply doubly transitive subgroup of $G$ on $\Omega$, this yields $G_2\leq S$, that is $G=S$.

If $S$ is of type (iv) in Result \ref{HC} then $S\cong {\rm{SU}}(3,2)$ and hence $\bar{S}\cong \PSU(3,2)$. Also, $\aut(\PSU(3,2))\cong {\rm{P\Gamma U}}(3,2)$, and every involution in ${\rm{P\Gamma U}}(3,2)\setminus
\PSU(3,2)$ has more than one fixed points. Again, $\bar{g}$ cannot be one of them, a contradiction.

\section{Examples for Theorem \ref{th1}}
\label{esempi}
For each group $G$ listed in Theorem \ref{th1} we exhibit an example of a plane curve with two different internal Galois points $P_1$ and $P_2$ both simple. These example arise from  automorphism groups satisfying (I),(II),(III) via Lemma \ref{profuka1}. We keep our notation used in Theorem \ref{th1}.

\subsection{Case (i)} We show that the curves on which $G$ acts naturally provide examples. All but the second examples on the Hermitian curve are known and they can be found in some recent papers of Fukasawa and his coauthors; see \cite{fuka7,fukaHiga1,FukaHiga2}.  We refer to those papers for the proofs of (I),(II),(III).
\subsubsection{Hermitian Curve}
 \label{exher}
 Let $q=p^h$. The Hermitian curve (also called the Deligne-Lusztig curve of unitary type) $\cX$ is  the non-singular plane curve $\cC$ of genus $\ha q(q-1)$ given by the affine equation $x^{q+1}+y^{q+1}+1=0$; see \cite[Section 12.3]{HKT}. Furthermore, $\PSU(3,q)$ is isomorphic to a subgroup $G$ of $\aut(\cX)\cong \PGU(3,q)$ which acts on the set $\Omega$ of all $\mathbb{F}_{q^2}$-rational points of $\cX$ as doubly transitive permutation group. Here $|\Omega|=q^3+1>2$, and the stabilizer of $P\in \Omega$ in $G$ contains a normal subgroup $N_P$ which acts on  $\Omega\setminus\{P\}$ as a sharply transitive permutation group, and $P$ is a Galois point of $\cC$ with Galois group $N_P$. For any two distinct points $P_1,P_2\in \Omega$, define $G_1=N_{P_1}$ and $G_2=N_{P_2}$. The subgroup $G=\langle G_1,G_2 \rangle$ is isomorphic $\PSU(3,q)$, and  $G$ is in turn the normal closure of $G_1$ in $G$.

 Another example arises from the Hermitian curve if $G$ is taken as the centralizer of an involution of $\aut(\cX)$ which is the subgroup of $\aut(\cX)$ preserving a chord $\ell$ of $\Omega$. Here $G\cong {\rm{SL}}(2,q)$ (and $\PSL(2,q)$ for even $q$). For any two distinct points $P_1,P_2\in \Omega\cap \ell$, define $G_i$ to be the subgroup fixing $P_i$. Then Conditions (II),(III) are satisfied. To show (I) the sequence of the ramification groups $G_1^{(i)}$ at $P_1$ is useful. From \cite[Lemma 12.1(e)]{HKT},  $G_1=G_1^{(0)}=G_1^{(1)}=\ldots =G_1^{(q)}$ whereas $G_1^{(q+1)}=\{1\}$. From the Hurwitz genus formula applied to $G_1$, $(q+1)(q-2)=2\gg(\cX)-2=q(2\gg(\cX/G_1)-2)+(q+1)(q-1)$, whence $\gg(\cX/G_1)=0$. Similarly, for $G_2$.
Moreover, $G=\langle G_1,G_2\rangle$, and $G$ is the normal closure of $G_1$ in $G$.

\subsubsection{Roquette Curve} Let $q=p^h>3$ with odd prime $p$. The Roquette curve $\cX$ is  the non-singular model of the irreducible (hyperelliptic) plane curve $\cC$ of genus $\ha (q-1)$ given by the affine equation $x^q-x=y^2$. Then either $\PSL(2,q)$ or ${\rm{SL}}(2,q)$ (according as $q\equiv 1 \pmod 4$ or $q\equiv -1 \pmod 4$) is isomorphic to a subgroup of $\aut(\cX)$ which acts on the set $\Omega$ of all $\mathbb{F}_{q^2}$-rational points of $\cX$ as a doubly transitive permutation group isomorphic to $\PSL(2,q)$.

\subsubsection{Suzuki Curve} Let $p=2,\ q_0=2^s,$ with $s\geq 0$ and $q=2q_0^2 = 2^{2s+1}$. The Suzuki curve (also called the Deligne-Lusztig curve of Suzuki type) $\cX$ is the non-singular model of the irreducible plane  curve $\cC$ of genus $q_0(q-1)$ given by the affine equation $x^{2q_0}(x^q+x)=y^q+y$; see \cite[Section 12.2]{HKT}. The Suzuki group $Sz(q)$ is isomorphic to a subgroup $G$ of $\aut(\cX)$ which acts on the set $\Omega$ of all $\mathbb{F}_{q^2}$-rational points of $\cX$. Here $|\Omega|=q^2+1>2$. 

\subsubsection{Ree Curve} Let $p=3$, $q=3q_0^2,$ with $q_0=3^s,\, s\geq 2.$ The Ree curve (also called the Deligne-Lusztig curve of Ree type) $\cX$ is the non-singular model of the irreducible  plane curve $\cC$ of genus $\frac{3}{2}q_0(q-1)(q+q_0+1)$ given by the  affine equation $y^{q^2}-[1+(x^q-x)^{q-1}]y^q+(x^q-x)^{q-1}y-x^q(x^q-x)^{q+3q_0}=0$; see \cite[Section 12.4]{HKT} Let $s\ge 2$. The Ree group $Ree(q)$ is isomorphic to a subgroup $G$ of $\aut(\cX)$ which acts on the set $\Omega$ of all $\mathbb{F}_{q^2}$-rational points of $\cX$ as a doubly transitive permutation group. 
\subsubsection{GK curve} Let $q=p^{3r},$ with $r\ge 1$. The GK curve is the non-singular model of the irreducible plane curve $\cC$ of genus $\ha\,(n^3+1)(n^2-2)+1$ given by the affine equation $y^{q+1}-(x^q+x)+(x^n+x)^{n^2-n+1}=0$ where $n=p^r$, see \cite{giuko}. Moreover, ${\rm{SU}}(3,n)$ is isomorphic to a subgroup of $\aut(\cX)$ which acts on the set $\Omega$ of the $n^3+1$  $\mathbb{F}_{q}$-rational points of $\cX$ as a doubly transitive permutation group. 

\subsection{Case (ii)} Let $p=3$. The Ree curve $\cX$ with $s=1$ provides an example. Indeed, $\rm{P\Gamma L}(2,8)$ is isomorphic to a subgroup $G$ of $\aut(\cX)$ which acts on the set $\Omega$ of the $28$ $\mathbb{F}_{q^2}$-rational points of $\cX$ as a doubly transitive permutation group.
\subsection{Cases (iii)} The basic tool is Result \ref{lem30oct2016}.
\subsubsection{Case {\rm{(iiia)}}} Let $m=p^h$. The rational curve $\cC$ with homogeneous equation $yz^{m-1}=x^m-xz^{m-1}$ is an example with $G\cong \AGL(1,m)$. To show this, observe that the non-singular points of $\cC$ defined over $\mathbb{F}_m$ are those lying on the $X$-axis, and they coincide with the points $P_u=(u,0,1)$ with $u\in \mathbb{F}_m$. For every non-zero $\lambda\in \mathbb{F}_m$ the transformation $w$ with $w(x)=\lambda x$, $w(y)=\lambda y$ is in $\aut(\cX)$ and preserves every line through $P_0$. They form a subgroup $G_1$ of order $m-1$ fixing $P_0$. Therefore, $P_0$ is a Galois point with Galois group $G_1$. The transformation $\tau$ with $\tau(x)=x-z$, $\tau(y)=y$ is in $\aut(\cX)$, and $G_2=\tau^{-1}G_1\tau$ is a subgroup of order $m-1$ fixing $P_1$. Therefore, $P_1$ is also a Galois point with Galois group $G_2$. Furthermore, $G_1\cap G_2=\{1\}$ and $G=\langle G_1,G_2\rangle\cong {\rm{AGL}}(1,m)$. Earlier reference for this example is \cite{fukahase}.
\subsubsection{Case {\rm{(iiib)}}}
 \label{m3}Let $p\neq 3$. The rational curve $\cC$ with  equation of degree $3$ provides an example with $G\cong {\rm{AGL}}(1,3)$. To show this, for a subgroup $G\cong {\rm{AGL}}(1,3)$, take an involution $\alpha\in G$. Let $P\in \cX$ be one of the fixed points of $\alpha$. Then the orbit $\Omega$ of $P$ in $G$ has size $3$. In $G$, take two distinct subgroups $G_1$ and $G_2$ of order $2$. Let $P_i$ with $i=1,2$ be the fixed point of $G_i$. Then conditions (I), (II) and (III) are satisfied. Therefore $P_i$ is an inner Galois point of $\cX$ with Galois group $G_i$.
\subsubsection{Case {\rm{(iiic)}}} Let $p\neq 2$. The quartic curve $\cC$ with homogeneous equation $x^2y^2+y^2z^2+z^2x^2=0$ is rational. For a primitive third root of unity $\varepsilon\in\mathbb{K}$, the cubic transformation $\alpha_1$ with $\alpha_1(x)=y,$ $\alpha_1(y)=z,$ $\alpha_1(z)=x$ is in $\aut(\cC)$ and fixes the point $P_1=(1:\varepsilon:\varepsilon^2)$. Also, the involution $\beta$ with $\beta(x)=x,$ $\beta(y)=-y,$ $ \beta(z)=z$  is in $\aut(\cC)$, and takes $P_1$ to the point $P_2=(1:-\varepsilon:\varepsilon^2)$. Therefore, $\alpha_2=\beta\alpha_1\beta\in \aut(\cC)$ is a cubic transformation such that $\alpha_2(x)=-y$, $\alpha_2(y)=-z$, $\alpha_2(z)=x$ and $\alpha_2(P_2)=P_2$. Let $G_i=\langle \alpha_i \rangle$  for $i=1,2$. Then $G=\langle G_1,G_2\rangle\cong {\rm{AGL}}(1,4)$, and Condition (I),(II), (III) are satisfied, and $|\Omega|=4>2$. Therefore,  $P_1$ and $P_2$ are Galois points with Galois groups $G_1$ and $G_2$, respectively. Plane quartic curves with two Galois points are investigated in \cite{fuka8}, where examples for Case (iiic) are also found.
\subsection{Cases (iv)} We show a general procedure relying on Lemma \ref{lemA10feb2019} which provides examples for $p\nmid m$. Let $\cE$ be an elliptic curve.  For a prime $r$ different from $p$, the translations in $\aut(\cE)$ associated to the $r$-torsion points together with the identity transformation form an elementary abelian subgroup $R$ of $\aut(\cE)$ of order $r^2$. In $\aut(\cE)$, the Jacobian subgroup $J(\cE)$  of $\aut(\cX)$ consisting of all translations of $\cE$ is abelian, and hence $R$ is the unique elementary abelian subgroup of $J(\cE)$. Since $J(\cE)$ is a normal subgroup of $\aut(\cE)$, this shows that $R$ is also a normal subgroup of $\aut(\cX)$. For a point $P_1\in \cE$ let $\Omega$ be the $R$-orbit of $P_1$, and $G_1$ the stabilizer of $P_1$ in $\aut(\cE)$.
For a non-trivial element $\alpha \in R$, the point $P_2=\alpha(P_1)$ is fixed by $G_2=\alpha^{-1}G_1\alpha$. Therefore, conditions (I) and (II) are satisfied. Moreover, Lemma \ref{lemA10feb2019} shows that no non-trivial element in $G_1$ fixes a point of $\Omega$ other than $P_1$. Therefore,  (III) holds with $\supp(D)=\Omega$ if and only if $|G_1|=r^2-1$.
If this is the case then
$G=\langle G_1,G_2 \rangle$ is sharply doubly transitive on $\supp(D)$, and, from Result \ref{sharply} and subsequent discussion, either $G\cong {\rm{AGL}}(1,r^2)$, or $G\cong {\rm{A\gamma L}}(1,r^2)$, or $G$ arises from an irregular nearfield.
This together with  Result \ref{autelliptic} provide an example with $m=4,9,25$; more precisely $\rm{AGL}(1,4)$ for $p\neq 2$, and ${\rm{A\gamma L}}(1,9)$, and $(C_5\times C_5)\rtimes {\rm{SL}}(2,3)$ for $p=2$. Therefore, Conditions (I), (II) and (III) are satisfied, and examples for (iva),(ivb),(ivc),(ivd),(ive) are obtained from (i),(ii) and of Proposition \ref{pro23janA}, respectively.

\subsection{Case (va)}
Let $p=2$. The $GK$ curve $\cC$ has genus $10$ and defined over $\mathbb{F}_{8}$ with homogeneous equation $z^9+x^8y+xy^8+(x^2y+xy^2)^3=0$.  $\cC$ has two Galois points $P_1=(0:1:0)$ and $P_2=(1:0:0)$  with Galois groups $G_1\cong G_2$. Here $G=\langle G_1,G_2\rangle\cong \SU(3,2)$ and $G_1$ is the Sylow $2$-subgroup of $P_1$ isomorphic to the quaternion group. Earlier reference of this example is \cite{FukaHiga2}.
\subsection{Case (vb)}
Let $p \neq 2,3$. The non-singular plane quartic $\cC$ of equation $X^4+Y^4+YZ^3=0$ has four internal Galois points, two of them are $P_1=(0:0:1)$  and $P_2=(0:-1:1)$. The group $G$ generated by the respective Galois groups is isomorphic to ${\rm{SL}}(2,3)$. Earlier reference of this example is \cite{my}.

\vspace{0.5cm}\noindent {\em Authors' addresses}:

\vspace{0.2cm}\noindent G\'abor KORCHM\'AROS, Stefano LIA and Marco TIMPANELLA\\ Dipartimento di
Matematica, Informatica ed Economia\\ Universit\`a degli Studi  della Basilicata\\ Contrada Macchia
Romana\\ 85100 Potenza (Italy).\\E--mail: {\tt
gabor.korchmaros@unibas.it}, {\tt stefano.lia@unibas.it} and {\tt marco.timpanella@unibas.it}.
\end{document}